\documentclass[11pt]{article}
\usepackage{amsmath,amssymb,amsthm}
\usepackage{graphicx}
\graphicspath{{./}{figures/}}
\usepackage[margin=1in]{geometry}
\usepackage{xcolor}
\usepackage[hidelinks]{hyperref}
\usepackage{mathtools}
\usepackage{tikz}
\usetikzlibrary{babel}

\newcommand{\D}{\mathbb D}
\newcommand{\C}{\mathbb C}

\newcommand{\T}{\mathbb T}

\newcommand{\Log}{\mathrm{Log}}
\newcommand{\Bee}{\mathcal B}

\newcommand{\M}{\mathcal M}
\newcommand{\cP}{\mathcal P}
\newcommand{\cS}{\mathcal S}
\newcommand{\cD}{\mathcal D}
\newcommand{\cA}{\mathcal A}
\newcommand{\cR}{\mathcal R}
\newcommand{\K}{\mathcal K}

\DeclareMathOperator{\Arg}{Arg}
\newcommand{\cl}{\overline}

\DeclareMathOperator{\Rea}{Re}
\DeclareMathOperator{\Ima}{Im}
\DeclareMathOperator{\dist}{dist}
\DeclareMathOperator{\wind}{wind}
\DeclareMathOperator{\interior}{int}
\DeclareMathOperator{\ext}{ext}

\DeclareMathOperator{\Int}{int}

\newtheorem{theorem}{Theorem}[section]
\newtheorem{lemma}[theorem]{Lemma}
\newtheorem{proposition}[theorem]{Proposition}
\newtheorem{corollary}[theorem]{Corollary}
\newtheorem{definition}[theorem]{Definition}
\theoremstyle{remark}
\newtheorem{remark}[theorem]{Remark}

\title{Smale's mean value constant and admissible univalent Blaschke pairs}
\author{Aneesh Jatar\and Tuen-Wai Ng}

\date{}
\begin{document}
\maketitle
\begin{center}
Department of Mathematics\\
The University of Hong Kong\\
Pokfulam Road, Hong Kong SAR\\[1em]
E-mail: \texttt{aneesh.jatar1@gmail.com}, \texttt{ntw@maths.hku.hk}
\end{center}
\begin{abstract}
Let $\mathcal{P}$ be the class of polynomials $P$ of degree at least $2$ with
$P(0)=0$, $P'(0)=1$, let $S(P)=\min_{P'(c)=0}|P(c)/c|$ and let
$\mathcal{M}=\sup_{P\in\mathcal{P}}S(P)$ be Smale's mean value constant. Smale proved that
$\M\le4$, observed that $\M\ge 1$, and asked whether $\M=1$. We show that $\mathcal{M}$ coincides with
$$
\mathcal{B}=\sup\frac1{B(0)},
$$
the supremum running over pairs $(W,B)$ in which $W$ is a normalised
univalent map of the standard unit disk $\D$ and $B$ is a finite Blaschke product with
$B(0)>0$ such that $WB$ is locally univalent in $\mathbb{D}$. Smale's question is
therefore a problem in univalent function theory. We also introduce the degree-restricted constants $\mathcal{B}_d$ for $B$ with $\deg B \le d$ and show that
$$1.1018586788\ldots=\sqrt{\frac{8081}{6656}} \le \mathcal{B}_1 <
\sqrt{\frac{7}{5}}=1.1832159566\ldots$$
This implies that $\mathcal{M}>1.1018$ and we also show that
there are polynomials of arbitrarily large degree with $S(P)>1.1018$. 
The polynomial construction behind these results also settles Problems 1A and 1B of Smale's 1981 paper. For $P\in\mathcal P$ and an integer $k\ge2$ let $S_k(P)=\min_{P'(c)=0}|P^{(k)}(0)/k!|^{1/(k-1)}|P(c)|$. Smale proved, by L\"owner's coefficient estimate for inverses of univalent functions, that $\sup_{P\in\mathcal P}S_k(P)\le C_k^{1/(k-1)}<4$, where $C_k$ is the $k$-th Catalan number and he asked whether the constant $4$ (Problem 1A), or the constant $2=C_2$ for $k=2$ (Problem 1B), can be reduced. We prove that $\sup_{P\in\mathcal P}S_k(P)=C_k^{1/(k-1)}$ for every $k\ge2$. Hence the constants $4$ and $2$ cannot be reduced and the answers to both problems are negative.

\end{abstract}


\section{Setting and results}

\subsection{Smale's constant}\label{sec:smale}

As a byproduct of Smale's study of the complexity of Newton's method \cite{Smale}, he introduced and studied the following constants.

Let $\cP_d$ be the class of polynomials $P$ of degree $d\ge2$ with
$P(0)=0$, $P'(0)=1$, and $\cP=\bigcup_{d\ge2}\cP_d$. For $P\in\cP$ put
$q(w)=P(w)/w$, a polynomial with $q(0)=1$ and $\deg q=\deg P-1\ge1$.
Smale's quantity is
$$
S(P):=\min_{P'(c)=0}\Bigl|\frac{P(c)}{c}\Bigr|=\min_{P'(c)=0}|q(c)| ,
\qquad
\M_d:=\sup_{P\in\cP_d}S(P),\qquad \M:=\sup_{P\in\cP}S(P).
$$
(The origin is not a critical point, since $P'(0)=1$.) The quantity
$S(P)$ is invariant under $P(w)\mapsto\lambda^{-1}P(\lambda w)$. Smale
\cite{Smale} proved $S(P)\le4$ for all $P\in\cP$, i.e.\ $\M\le4$. He also noticed that $P(w)=w-w^d/d$ gives $S(P)=(d-1)/d$, and asked whether $\M=1$ or
even $\M_d=(d-1)/d$. This is Smale's mean value conjecture, one of the
three minor problems in his problem list for the next century \cite{Smale2000}.
The conjecture is true for $d\le4$ and for special classes of
polynomials; the degree-dependent refinements
$\M_d\le4^{(d-2)/(d-1)}$ \cite{BMN} and $\M_d\le4-c/\sqrt d$
\cite{Crane} tend to $4$ as $d\to\infty$. So $\M\le4$ is still the best
known upper bound for $\M$ itself. See \cite{Crane,NgZhang,JN26, FLW} for
surveys of the literature and the latest results.

In \cite{Lean} a pre-release version of GPT-6 Astra produced a
machine-checked (Lean~4/Mathlib) proof that the constant-$1$ form of the
conjecture is false: {\it there is a polynomial $P$ of large, unspecified
degree with $S(P)>1$}. In \cite{Lean}, the
counterexample is built in four steps: i) a locally univalent function of
the form $(z-z_0)B_r(W(z))/B_r(W(z_0))$ is constructed on the quartic
lemniscate $K=\{|z^4+64i|\le\sqrt{4097}\}$, where $W:K\to\cl\D$ is a
conformal map and $B_r$ is a Blaschke factor; ii) it is approximated by a
polynomial $f_0$; and iii) the primitive of a high even power of the quartic $q(z)=z^4+64i$ is added to $f_0$, which forces all critical points of the polynomial $P_n=f_0 + \int q^{2n}$ into a thin collar around
$\partial K$; iv) rescale $P_n$ to get a polynomial $P \in \cP$ so that $S(P)>1$. The same argument can be further developed to give for example $S(P)>1+2^{-35}$, 
but it is hard to push this approach to reach
a lower bound like $1.1018$.

Motivated and based upon the GPT-6 Astra construction, in the present paper, we give two reformulations of Smale's mean value problem by introducing two new constants $\cA$ and $\Bee$ which arise (respectively) from two (equivalent) extremal problems for pairs of the form $(f,\Omega)$ and $(W,B)$, where $f(0) = 0, f'(0) = 1$ is locally univalent in a neighbourhood of $\overline{\Omega}$ and $\Omega$ is a bounded, analytic Jordan domain containing the origin, while $W$ is a univalent function on the unit disk fixing the origin, and $B$ is a finite Blaschke product for which $WB$ is locally univalent on the unit disk. We show that $\cA=\M=\Bee$. The definitions of $\cA$ and $\Bee$ will be given in Section \ref{sec:AM} and Section \ref{sec:WB} respectively.  

Smale's paper \cite{Smale} also contains a second family of extremal problems for critical values, which arises from the Taylor expansion of $P$ at the current point of the Newton iteration. For $P\in\cP$ and an integer $k\ge2$ let
$$
S_k(P):=\min_{P'(c)=0}\Bigl|\frac{P^{(k)}(0)}{k!}\Bigr|^{\frac1{k-1}}|P(c)|,\qquad
\K_{d,k}:=\sup_{P\in\cP_d}S_k(P),\qquad
\K_k:=\sup_{P\in\cP}S_k(P),\qquad
\K:=\sup_{k\ge2}\K_k
$$
(see Section~\ref{sec:smale1A} for the equivalent unnormalized form). Using L\"owner's theorem on the coefficients of inverses of univalent functions, Smale proved that $\K_k\le C_k^{1/(k-1)}<4$, where $C_k=\frac{(2k)!}{k!\,(k+1)!}$ is the $k$-th Catalan number, i.e.\ the modulus of the $k$-th Taylor coefficient of the inverse Koebe function; hence $\K\le4$. He observed that $\K\ge1$ and asked in Problem 1A of \cite{Smale} whether the constant $4$ can be reduced, possibly to $1$, and in Problem 1B whether the constant $2=C_2$ for $k=2$ can be reduced to $\frac{d-1}{2d}$ for polynomials of degree $d$. For the top coefficient $k=d$, Dubinin \cite{Dubinin2006} proved the sharp result $\K_{d,d}=(d-1)d^{-d/(d-1)}$, so that Smale's example $w-w^d/d$ is extremal for $k=d$, while the lower bound of $\frac{d-1}{2d}$ for $K_{d,2}$ is known to be false for infinitely many degrees $d$ (see the thesis of Choi \cite{choi} and also Cheung \cite{cheung2011}). On the other hand, despite the bound of $\frac{d-1}{2d}$ being false for infinitely many degrees, it is not known whether the constant $2$ can be reduced at all. We will refer to this more general question again as ``Problem 1B'' in the remainder of this paper. We observe that the Runge--Hilbert construction which gives $\cA\le\M$ (Proposition~\ref{prop:completion}) works equally well for the quantities $S_k$ and, when applied to the inverse Koebe function, produces the reverse inequality $\K_k\ge C_k^{1/(k-1)}$. Thus $\K_k=C_k^{1/(k-1)}$ for every $k\ge2$, $\K_2=2$ and $\K=4$. In conclusion, Smale's constants $4$ and $2$ cannot be reduced, and Problems 1A and 1B are answered in the negative.

\subsection{Admissible univalent Blaschke pairs}\label{sec:WB}

Write $\D=\{z\in\C:|z|<1\}$, $\T=\partial\D$,
$D(w,\rho)=\{\zeta:|\zeta-w|<\rho\}$, and let $\cS$ denote the class of
functions $W$ holomorphic and univalent in $\D$ with $W(0)=0$ and $W'(0)=1$.

A \emph{finite Blaschke product} is a function
$$
B(z)=\lambda\prod_{k=1}^{n}\frac{z-a_k}{1-\cl{a_k}z},
\qquad |\lambda|=1,\quad a_k\in\D,\quad n\ge 0 ,
$$
zeros being repeated according to multiplicity and $n=\deg B$. We shall
use the following standard facts: $B$ is holomorphic on a neighbourhood
of $\cl\D$ (its poles $1/\cl{a_k}$ lie outside $\cl\D$), $|B|=1$ on
$\T$, and $|B|<1$ in $\D$ if $n\ge1$. If no $a_k$ is $0$, the
normalization $B(0)>0$ fixes the unimodular constant, namely
$\lambda=\prod_k(-\cl{a_k}/|a_k|)$, and then
\begin{equation}\label{eq:B0}
B(0)=\prod_{k=1}^{n}|a_k| .
\end{equation}

A pair consisting of a univalent function $W$ and a finite Blaschke product $B$ is called a \emph{univalent Blaschke pair} or simply a \emph{pair} $(W,B)$. To reformulate Smale's mean value constant $\M$, we also introduce the following
\begin{definition}\label{def:admissible}
A pair $(W,B)$ is (normalized) \emph{admissible} if
\begin{enumerate}
\item $W\in\cS$;
\item $B$ is a finite Blaschke product with $B(0)>0$;
\item $F:=WB$ satisfies $F'(z)\ne0$ for all $z\in\D$.
\end{enumerate}
The constant $B\equiv1$ is allowed. We set
$$
\Bee:=\sup\Bigl\{\frac{1}{B(0)}:(W,B)\ \text{admissible}\Bigr\},
\qquad
\Bee_d:=\sup\Bigl\{\frac{1}{B(0)}:(W,B)\ \text{admissible},\ \deg B\le d\Bigr\}.
$$
An admissible pair is \emph{regular} if $W$ extends to a univalent
function on a neighbourhood of $\cl\D$ and $F'\ne0$ on $\cl\D$.
\end{definition}

Since $(W,1)$ is admissible for every $W\in\cS$, $\Bee\ge\Bee_d\ge 1$.
Note that $B(0)\neq 0$ is forced by condition 3 because a zero of $B$ at the
origin would make $F$ vanish to order $\ge2$ there, so $F'(0)=0$.
For the same reason all zeros of $B$ are simple.
Observe also that $F'(0)=W'(0)B(0)=B(0)$. The normalizations $W'(0)=1$
and $B(0)>0$ are harmless: for any univalent $W$ with $W(0)=0$ and any
finite Blaschke product $B$ with $(WB)'\ne0$ on $\D$, the pair
$(W/W'(0),\,B\,|B(0)|/B(0))$ is admissible with the same value
$1/|B(0)|$. We will thus refer to any pair $(W,B)$ with $W(0) = 0$ univalent on $\mathbb{D}$, $B$ a finite Blaschke product, such that $WB$ is locally univalent on $\mathbb{D}$ as \textbf{admissible} in the sequel.

Throughout we use the notation
$$
F=WB,\qquad g=\frac{zW'}{W},\qquad \Phi=\frac{zF'}{F}=g+\frac{zB'}{B}.
$$
Since $W(z)=z+O(z^2)$ and $W\neq0$ on $\D\setminus\{0\}$, the function
$g$ is holomorphic on $\D$ with $g(0)=1$; $\Phi$ is meromorphic on $\D$,
$\Phi(0)=1$, and its poles in $\D$ are exactly the distinct zeros of $B$
(all simple). Moreover, $\Phi$ has no zeros in $\D$. 

\subsection{Main results}

\begin{theorem}\label{thm:main}
$\displaystyle \M\ =\ \Bee .$
\end{theorem}

The inequality $\M\le\Bee$ is proved in Section~\ref{sec:poly-to-pair}
by a Riemann map of a sublevel component of $P(w)/w$. The inequality
$\Bee\le\M$ is then proved in Section~\ref{sec:B=M}, through the intermediate constant $\cA$ of Section~\ref{sec:AM}: we show $\Bee\le\cA\le\M$, so that in fact $\M=\Bee=\cA$ (Theorem~\ref{thm:BM}).

\begin{theorem}\label{thm:upperN}
For every admissible pair $(W,B)$ one has $B(0)\ge m(W)\ge \tfrac14$, where $m(W)=\dist(0,\partial W(\D))$. Hence $\Bee\le4$, and by Theorem~\ref{thm:main} this recovers Smale's theorem $\M\le4$. Moreover, for every $d\ge1$,
$$\Bee_d <  4-\frac{529}{576}\frac{(\log 2)^2}{d^2},$$
and
$$1.1018586788\ldots=\sqrt{\frac{8081}{6656}} \le \Bee_1 <
\sqrt{\frac{7}{5}}=1.1832159566\ldots$$
Hence $\M > 1.1018$, which improves the result $\M>1$ in \cite{Lean}, and $\M_d>1.1018$ for all large $d$.
\end{theorem}

\begin{theorem}\label{thm:Kintro}
For every integer $k\ge2$,
$$
\K_k=\sup_{P\in\cP}S_k(P)=C_k^{\frac1{k-1}},\qquad C_k=\frac{(2k)!}{k!\,(k+1)!},
$$
and for every $A<C_k^{1/(k-1)}$ there are polynomials $P\in\cP$ of arbitrarily large degree with $S_k(P)>A$. Consequently $\K_2=2$ and $\K=4$; in particular Smale's Problems 1A and 1B in \cite{Smale} have negative answers. Moreover for Problem 1B, $K_{d,2}$ can be made arbitrarily close to $2$ for sufficiently large degrees $d$.
\end{theorem}

 The identification $\M=\Bee$ recasts Smale’s mean value problem entirely within univalent function theory and yields the improved quantitative lower bound $\M>1.1018$ and hence a new disproof of Smale's mean value conjecture.
 This identification also makes available the full toolkit of univalent function theory (e.g., Grunsky inequalities) for studying the new questions of determining the exact values of the constants $\Bee_d$ and 
whether $\Bee=4$. The same approach also suggests a systematic route for translating other classical extremal problems involving polynomials into questions about pairs of univalent maps and Blaschke products. In fact, Sheil-Small has already suggested this translation in 
 \cite[Chapter 10]{Sheil-Small} and our reformulation is linked to a special case of Sheil-Small's translation. Theorem~\ref{thm:Kintro} is an instance of this route in the opposite direction: the extremal function of univalent function theory (the Koebe function) is converted into extremal polynomials for Smale's constants $\K_k$.
 
 \medskip
The paper recalls Smale’s constant $\M$ and introduces the reformulation $\Bee$ via admissible univalent Blaschke pairs in Section~\ref{sec:smale}--\ref{sec:WB}. Section~\ref{sec:regular} contains an auxiliary regularization result (not needed for the main theorem). The inequality $\M\le\Bee$ is established in Section~\ref{sec:poly-to-pair} by a Riemann-map construction from polynomials. In Section~\ref{sec:B=M} the companion constant $\cA$ is defined, the elementary inequality $\Bee\le\cA$ is proved, and the only analytic step of the paper, $\cA\le\M$, is obtained from Runge approximation and Hilbert's Lemniscate Theorem; this gives $\M=\Bee=\cA$ and the existence of infinitely many degrees $d$ with $\M_d>1.1018$. Section~\ref{sec:smale1A} recalls Smale's constants $\K_{d,k}$, $\K_k$, $\K$ and his Problems 1A and 1B, and combines a variant of Proposition~\ref{prop:completion} with L\"owner's theorem to prove Theorem~\ref{thm:Kintro}. The degree-one lower bound appears in Section~\ref{sec:quartic} via an explicit quartic and Dieudonn\'e's criterion, while the upper bounds $\Bee\le4$, $\Bee_d\le4-O(1/d^2)$ and $\Bee_1\le\sqrt{7/5}$ are obtained in Section~\ref{sec:four} from the Koebe theorem, the Schwarz lemma, Pick's omitted-value estimate and Grunsky’s inequality. Section~\ref{sec:every} proves existence of admissible pairs of every degree, a necessary and (essentially) sufficient condition on $W$ for the existence of a nonconstant factor $B$, attainment of each $\Bee_d$, and the strict inequalities of Theorem~\ref{thm:upperN}. Concluding remarks, open problems and dual constants are given in Section~\ref{sec:cr}.

\subsection{Standard tools}\label{sec:tools}

We collect the classical results used below, with references.

\begin{itemize}
\item[(T1)] \emph{Koebe's one-quarter theorem} \cite[Chapter~2]{Duren}.
For $W\in\cS$, $W(\D)\supset D(0,1/4)$.
\item[(T2)] \emph{Analytic boundaries} \cite[Chapter~3]{Pommerenke}.
If $\Omega$ is a Jordan domain bounded by a real-analytic Jordan curve,
the Riemann map $\D\to\Omega$ extends to a conformal (univalent
holomorphic) map of a neighbourhood of $\cl\D$; in particular it is
Lipschitz on $\cl\D$ and maps $\T$ onto $\partial\Omega$.
\item[(T3)] \emph{Runge's theorem} \cite[Ch.~13]{Rudin}. If $K\subset\C$
is compact with connected complement and $p$ is holomorphic on an open
set containing $K$, there are polynomials converging to $p$ uniformly on
$K$. By the Cauchy estimates the derivatives then converge uniformly on compact subsets of the interior of $K$.
\item[(T4)] \emph{Proper self-maps of the disk} \cite[p.211-213]{Re}.  If $B:\D\to\D$ is
holomorphic and proper, then $B$ is a finite Blaschke product.
\item[(T5)] \emph{Hilbert's Lemniscate Theorem} \label{thm:hilbert} \cite{Hilbert,NT}.
Let $U$ be a closed topological annulus in $\C$. There is a
nonconstant polynomial $q$ such that the entire lemniscate
$$
\Gamma=\{w\in\C:|q(w)|=1\}
$$
is an analytic Jordan curve contained in $U$ and separating its
two boundary components. Applying the theorem to a smaller annulus, one may require
$\Gamma\subset\Int U$.

\item[(T6)] \emph{Riemann--Hurwitz count for a Jordan domain.} Let $C$ be a Jordan domain bounded by an analytic Jordan curve $\gamma$, and let $h$ be holomorphic on a neighbourhood of $\cl C$ with $|h|=1$ on $\gamma$, $|h|<1$ in $C$ and $h'\ne0$ on $\gamma$. If $h$ has $m$ zeros in $C$ (with multiplicity), then $h'$ has exactly $m-1$ zeros in $C$.

\emph{Proof.} $h|_\gamma:\gamma\to\T$ is a local homeomorphism, hence a covering of some degree $m'$, orientation preserving because $|h|<1$ on the left of $\gamma$. By applying the argument principle to $h$, $m'=m$. If $\gamma(s)$ is a positive parametrization, $\arg\frac{d}{ds}h(\gamma(s))=\arg h'(\gamma(s))+\arg\gamma'(s)$ increases by $2\pi m$ along $\gamma$, while $\arg\gamma'$ increases by $2\pi$ (Hopf Umlaufsatz). Hence $\Delta_\gamma\arg h'=2\pi(m-1)$, and the argument principle applied to $h'$ gives the claim. \qed

In particular, if $q$ is a polynomial of degree $m\ge1$ whose lemniscate $\Gamma=\{|q|=1\}$ is an analytic Jordan curve, then $|q|>1$ on $\ext\Gamma$ (as $|q|\to\infty$ and $|q|\ne1$ on the connected set $\ext\Gamma$), so all $m$ zeros of $q$ lie in $\interior\Gamma$ and $|q|<1$ there; $q'\ne0$ on $\Gamma$ (a critical point on $\Gamma$ would make $\{|q|=1\}$ branch there); and by the count above all $m-1$ critical points of $q$ lie in $\interior\Gamma$.

\end{itemize}

\noindent
{\it Acknowledgments.} The authors introduced the definitions of the constants $\cA$, $\Bee$, $\Bee_d$, $\Bee^\vee$ and $\Bee^\vee_d$, the identifications $\M=\cA=\Bee$ and the specific classes of the admissible pairs $(W,B)$ used. They also provided solution of Smale's Problems 1A and 1B given in Section~\ref{sec:smale1A} (Theorem~\ref{thm:Kk}). GPT-6-Astra and Grok 4.6 were utilized to conduct numerical experiments and generate preliminary draft proofs via targeted prompting. The authors then collaboratively verified, significantly simplified, and rewrote all proofs in their own words.

\section{Reduction to regular admissible pairs}\label{sec:regular}

The following lemma is  needed for the proof of Theorem~\ref{thm:main} (see  Section~\ref{sec:B=M}) and Section~\ref{sec:every}. It shows that the extremal problems defining $\Bee$ and $\Bee_d$ may be restricted to regular pairs.

\begin{lemma}\label{lem:regularization}
Let $(W,B)$ be admissible with $B$ nonconstant, and let $b=B(0)$. For
all $t<1$ sufficiently close to $1$, there is a regular admissible pair
$(W_t,B_t)$ with
$$
\deg B_t=\deg B,\qquad B_t(0)=\frac bt .
$$
Consequently the suprema $\Bee$ and $\Bee_d$ are unchanged if they are
taken over regular admissible pairs only.
\end{lemma}

\begin{proof}
Let $n=\deg B\ge1$. By \cite[Section~2]{NgZhang} (see also \cite[\S11.3]{Sheil-Small}), $B'$ has exactly $n-1$ zeros in $\D$, counted with multiplicity, and $|B|<1$ in $\D$. Hence
$$
T:=\max\{|B(c)|:\ c\in\D,\ B'(c)=0\}<1
$$
(with $T=0$ if $n=1$). Fix $t$ with $T<t<1$ and let
$$
J_t:=\{z\in\D:|B(z)|<t\},\qquad \Lambda_t:=\{z\in\D:|B(z)|=t\}.
$$

Using the covering space argument in the proof of \cite[Theorem~3]{NgZhang}, there is a biholomorphism
$$
\phi:\{t^{1/n}<|w|<1\}\to\{z\in\D:t<|B(z)|<1\}
$$
with $B\circ\phi(w)=\lambda w^n$, $|\lambda|=1$. Because $B'\ne0$ on $\Lambda_t$, $\phi$ extends analytically and injectively across the inner circle. So $\Lambda_t=\phi(\{|w|=t^{1/n}\})$ is a single analytic Jordan curve and hence $J_t$ is an analytic Jordan domain containing $0$ and all zeros of $B$.

Now let $\psi_t:\D\to J_t$ be the Riemann map with $\psi_t(0)=0$, $\psi_t'(0)>0$. By (T2) it extends to a conformal map of a neighbourhood of $\cl\D$, with $\psi_t(\cl\D)=\cl{J_t}\subset\D$. Set
$$
W_t:=\frac{W\circ\psi_t}{\psi_t'(0)},\qquad B_t:=\frac{B\circ\psi_t}{t}.
$$
Then $W_t$ is univalent on a neighbourhood of $\cl\D$ with $W_t(0)=0$, $W_t'(0)=1$, so $W_t\in\cS$. Since $|B|=t$ on $\partial J_t$, $B:J_t\to D(0,t)$ is proper, so $B_t:\D\to\D$ is a proper holomorphic self-map and by (T4) a finite Blaschke product; its degree is the number of zeros of $B$ in $J_t$, namely $n$, and $B_t(0)=b/t>0$. Finally
$$
(W_tB_t)'(z)=\frac{\psi_t'(z)}{t\,\psi_t'(0)}\,(WB)'(\psi_t(z))\ne0\qquad(z\in\cl\D),
$$
because $(WB)'\ne0$ on $\D\supset\cl{J_t}$. Hence $(W_t,B_t)$ is a regular admissible pair with the stated properties, and $1/B_t(0)=t/b\to1/b$ as $t\uparrow1$, so the suprema are unchanged.
\end{proof}

\section{From polynomials to admissible pairs}\label{sec:poly-to-pair}

For $P\in\cP$, $q=P(w)/w$ and $R>1$ let
$$
\Omega_R(P):=\text{the connected component of }\{w\in\C:|q(w)|<R\}
\text{ containing }0
$$
(it contains $0$ because $q(0)=1<R$).

\begin{proposition}\label{prop:poly-to-pair}
Let $P\in\cP_d$ with $S(P)>1$. Suppose $1<R<S(P)$, then there is an
admissible pair $(W,B)$ with
$$
B(0)=\frac1R,\qquad 1\le\deg B\le d-1 .
$$
Consequently for every $d\ge2$, $\Bee_{d-1}\ge S(P)$ for every $P\in\cP_d$ and hence
$\M_d\le\Bee_{d-1}$  and $\M\le\Bee$.
\end{proposition}

\begin{proof}
Write $\Omega=\Omega_R(P)$. Since every critical point $c$ of $P$ satisfies
$|q(c)|\ge S(P)>R$, $c\notin q^{-1}(D(0,R))$. It then follows from $\Omega \subset q^{-1}(D(0, R))$ that
\begin{equation}\label{eq:noncrit}
P'(w)=q(w)+wq'(w)\ne0\qquad(w\in\Omega).
\end{equation}

$\Omega$ is open and connected, and bounded
because $|q(w)|\to\infty$ as $|w|\to\infty$. It is also simply connected because if 
$\gamma\subset\Omega$ is a Jordan curve with interior $I$, the maximum
principle gives $|q|<R$ on $I$, so $I\cup\gamma$ is a connected subset
of $\{|q|<R\}$ meeting $\Omega$, hence contained in $\Omega$.

Since $q$ is a polynomial, $q:\Omega\to D(0,R)$ is a proper map. We now construct an admissible univalent Blaschke pair. Let $\phi:\D\to\Omega$ be the Riemann map with
$\phi(0)=0$ and $\phi'(0)=\kappa>0$, and define
$$
W=\frac{\phi}{\kappa},\qquad B=\frac{q\circ\phi}{R}.
$$
Then $W\in\cS$. The map $B$ is a proper holomorphic self-map of $\D$
(composition of the conformal map $\phi$ with the proper map $q/R$), so
by (T4) it is a finite Blaschke product with $B(0)=q(0)/R=1/R>0$. Its
degree equals the number of zeros of $q$ in $\Omega$, counted with
multiplicity, which is at most $\deg q=d-1$ and at least $1$ (a proper map onto $D(0,R)$ is surjective, so $q$ vanishes somewhere in $\Omega$). Finally
$$
(WB)'(z)=\frac{\phi'(z)}{\kappa R}\bigl(q(\phi(z))+\phi(z)q'(\phi(z))\bigr)
=\frac{\phi'(z)}{\kappa R}\,P'(\phi(z))\ne0\qquad(z\in\D)
$$
by \eqref{eq:noncrit}. Hence $(W,B)$ is admissible and $\Bee_{d-1}\ge R$.
Letting $R\uparrow S(P)$ gives $\Bee_{d-1}\ge S(P)$ if $S(P)>1$. If $S(P)\le1$, this
holds trivially since $\Bee_{d-1}\ge1$. Taking suprema over $P$ yields
$\M_d\le\Bee_{d-1}$ and $\M\le\Bee$.
\end{proof}

\section{$\Bee = \M$}\label{sec:B=M}
\begin{theorem}\label{thm:BM}
$\Bee=\M$.
\end{theorem}

The above theorem will follow immediately from  Proposition \ref{prop:poly-to-pair} and the below Proposition \ref{prop:BM}.

\begin{proposition} \label{prop:BM}
For every admissible pair $(W,B)$ and every positive real number $A < \frac{1}{B(0)}$, there is a polynomial $P \in \mathcal{P}$ such that $S(P) > A$. Consequently $$\Bee \leq \M, \ \Bee_{d} \leq \M \quad (d \geq 0) $$
\end{proposition}

\begin{proof} Given an admissible pair $(W,B)$ with $0< A < \frac{1}{B(0)}$. If $B\equiv1$, then $A<1$ and
$P(w)=w-w^d/d$ gives $S(P)=(d-1)/d>A$ for all large $d$.
So we will assume $B$ is non-constant. 
For non-constant $B$, Lemma~\ref{lem:regularization} provides, for every $t<1$ sufficiently close to $1$, a regular admissible pair $(W_t,B_t)$ with $B_t(0)=B(0)/t$; choosing $t$ so close to $1$ that $t/B(0)>A$, we still have $A<1/B_t(0)$. Hence we will further assume that $(W,B)$ is regular.

We now show that $\Bee \leq \M$. 
The essential idea is to "invert" the procedure used to produce a (regular) admissible pair from a polynomial of Section \ref{sec:poly-to-pair}, i.e. given a regular admissible pair $(W,B)$ (with $B(0) > 0$ and $W'(0) > 0$ univalent, not necessarily Schlicht), we assume for the moment, that the pair $(W,B)$ was produced via the procedure of Section \ref{sec:poly-to-pair}, which means that $W$ is the (normalized) Riemann map fixing the origin from $\mathbb{D}$ to the component of $$\{ |\frac{P(z)}{z}| < \frac{1}{B(0)} \}$$ containing the origin, for some polynomial $P \in \mathcal{P}_{d}$ and $$B = \frac{P \circ W}{B(0)^{-1}W}$$

Our first task is to recover the polynomial $P$ from such a (regular) admissible pair. From the above representation we necessarily have $$P(z) = \frac{z \cdot B \circ W^{-1}(z)}{B(0)}$$

Motivated by this, for any regular admissible pair (with the normalizations above), we define $$\tilde{P}(z) \coloneq  \frac{z \cdot B \circ W^{-1}(z)}{B(0)}
.$$

One can check directly from the above definition and the regularity of $(W,B)$ that $\tilde{P}'$ is nowhere vanishing in a neighbourhood of $W(\overline{\mathbb{D}})$, and that $\tilde{P}(0) = 0$ and $\tilde{P}'(0) = 1$. Moreover, $$|\frac{\tilde{P}(z)}{z}|$$ can be made arbitrarily close to $$\frac{1}{|B(0)|}$$ on sufficiently thin annular (i.e. doubly connected) neighbourhoods of the analytic Jordan curve $W(\partial \mathbb{D})$ bounded by the analytic Jordan curves $$W((1 - \epsilon) \partial \mathbb{D})$$ and $$W((1+ \epsilon) \partial \mathbb{D})$$ where $\epsilon > 0$ is chosen sufficiently small. We will refer to such an annular neighbourhood as an $\epsilon$-annular neighbourhood of $W(\partial \mathbb{D})$. If we now knew that $\tilde{P}$ (or an appropriate modification of $\tilde{P}$) were a polynomial, and that all of $\tilde{P}$'s critical points (or the critical points of an appropriate modification of $\tilde{P}$) were contained in the above annular neighbourhood of $$W(\partial \mathbb{D})$$ then this would be enough to recover $$\Bee \leq \M$$

This brings us to the main two issues. The first is that $\tilde{P}$ \textbf{need not be a polynomial} (since we have started from an arbitrary regular pair $(W,B)$ which need not be produced from the procedure of Section \ref{sec:poly-to-pair}), and the second is that even if we knew that $\tilde{P}$ were a polynomial, we have no control over how close its critical points lie to $W( \partial \mathbb{D})$. We now proceed to resolve both of these issues.

The first issue is resolved essentially immediately via Runge approximation of $\tilde{P}$ on the compact set $K=W(\cl{D(0,1+2\delta)})$, i.e. the closure of the Jordan domain bounded by the analytic Jordan curve $W((1 + 2\delta)\partial \mathbb{D})$, whose complement is connected. More precisely, Runge's theorem (T3) produces a sequence of polynomials $P_n : n = 1,2, \ldots$ which converge to $\tilde{P}$ in the supremum norm on $K$. It follows that  $P_n(0)\to\tilde P(0)=0$ and $P_n'(0)\to\tilde P'(0)=1$.  Replacing $P_n$ by $(P_n-P_n(0))/P_n'(0)$ we may assume without loss of generality that they satisfy $$P_n(0) = 0, \qquad P_n'(0) = 1$$ for each $n = 1,2, \ldots$. Moreover, since $$\tilde{P}'$$ is nowhere vanishing in a neighbourhood of $$W(\overline{\mathbb{D}})$$ By (T3) and the Hurwitz's theorem, for all $n$ sufficiently large, we have $P_n'$ is nowhere vanishing in the closure of some fixed neighbourhood of $W(\overline{\mathbb{D}})$, bounded by the analytic Jordan curve $$W((1 + \delta)\partial \mathbb{D}) $$ By taking $n$ sufficiently large, and $0 < \epsilon < \delta$ sufficiently small, we then have that $$|\frac{P_n(z)}{z}|$$ can be made arbitrarily close to $$\frac{1}{B(0)}$$ on the $\epsilon$-annular neighbourhood $N$ of $W(\partial \mathbb{D})$. Since $A<1/B(0)$, we may assume that 
\begin{equation}\label{eq:annN}
\Bigl|\frac{P_n(z)}{z}\Bigr|>A + c\qquad\text{for all }z\in N
\end{equation}
and $P_n'\neq0$ on the closed Jordan domain bounded by $W((1+\delta)\partial\mathbb{D})$, and $c > 0$ is fixed.

We now proceed to resolve the second issue via a modification of $P_n$ based on the original GPT-6 Astra construction. Let $q$ be a polynomial such that $\Gamma = \{ |q(z)| = 1 \}$ is an analytic Jordan curve contained in $N_\epsilon$, the interior of the annular region bounded by $W((1 - \epsilon) \partial \mathbb{D})$ and $W((1 + \epsilon) \partial \mathbb{D})$ and separating its boundary components (such a polynomial exists by Hilbert's lemniscate theorem (T5), applied to $\cl{N_\epsilon}$). The connected compact set $W(\cl{D(0,1-\epsilon)})$ is disjoint from $\Gamma$ and contains the inner boundary curve $W((1-\epsilon)\partial\mathbb D)$, which $\Gamma$ separates from $W((1+\epsilon)\partial\mathbb D)$; hence $W(\cl{D(0,1-\epsilon)})\subset\operatorname{int}\Gamma$ and in particular $0\in\operatorname{int}\Gamma$. By (T6), all zeros and critical points of $q$ lie in $\operatorname{int}\Gamma$, $|q|<1$ on $\operatorname{int}\Gamma$ and $|q|>1$ on $\operatorname{ext}\Gamma$. In particular $|q(0)|<1$ and all of $q$'s critical values have modulus strictly less than $1$. Let $T<1$ denote the maximum modulus of a critical value of $q$ ($T=0$ if $\deg q=1$). We now define $$H_{m,n}(z) = P_n(z) + z q^{m}(z),\qquad m =1,2,\ldots$$ Fix $r\in(0,1-T)$ so small that
$$
\{z:\ 1-r\le|q(z)|\le1+r\}\subset N_\epsilon 
$$
and all the critical values of $q$ have modulus less than $1-r$. 

We claim that all the critical points of $H_{m,n}$ lie in the annular region $\{1-r<|q|<1+r\}$ bounded by $$\Gamma^{-}_{r} = \{ |q(z)| = 1 - r \}$$ and $$\Gamma^{+}_{r} = \{ |q(z)| = 1 + r \}$$ given $m$ is sufficiently large.

Suppose for the sake of a contradiction, this were not the case, then (after passing to a subsequence and suppressing the relabeling) we may assume without loss of generality that either there exists a sequence of points $b_{m} \in \mathrm{int} \ \Gamma_{r}^{-} : m = 1,2, \ldots$ for which $$H_{m,n}'(b_m) = P_n'(b_m) + q^{m}(b_m) + b_m m q^{m-1}(b_m) q'(b_m) = 0$$ or there exists a sequence of points $b_m \in \mathrm{ext} \ \Gamma_{r}^{+}$ for which $$H_{m,n}'(b_m) = P_n'(b_m) + q^{m}(b_m) + b_m m q^{m-1}(b_m) q'(b_m) = 0$$

In the former case, we may assume without loss of generality (after passing to a subsequence), that $$b_m \rightarrow b \in \overline {\mathrm{int} \ \Gamma^{-}_{r}}$$ and necessarily then $P_n'(b) = 0$. This yields a contradiction as $P_n'$ is nowhere vanishing in the closure of the Jordan domain bounded by $W((1 + \delta) \partial \mathbb{D})$, and $$b \in \overline {\mathrm{int} \ \Gamma^{-}_{r}} \subset \mathrm{int} \ W ((1 + \epsilon) \partial \mathbb{D}) \subset \mathrm{int} \ W( (1 + \delta) \partial \mathbb{D})$$

In the latter case, we may assume without loss of generality (after passing to a subsequence), that $$b_m \rightarrow b \in \overline{\mathrm{ext} \ \Gamma^{+}_{r}}$$ Suppose $b \neq \infty$. Then $$|P_n'(b_m)| =|q^{m-1}(b_m)| |m b_m q'(b_m) + q(b_m)| \geq (1+r)^{m-1} |m b_m q'(b_m) + q(b_m)|$$

On the other hand, since $q'(b) \neq 0$, as all zeros and critical points of $q$ lie in the interior of $\Gamma$, and the interior of $\Gamma$ contains the origin (so in particular, $b \neq 0$) we conclude that $$|m b_m q'(b_m) + q(b_m)| \sim |m b q'(b)|$$ as $m \rightarrow \infty$, so that $$|P_n'(b_m)| \gtrsim (1+r)^{m-1} m |b| q'(b) $$ as $m \rightarrow \infty$, and in particular $$P_n'(b) = \infty$$ Since $b$ is a finite point and $P_n'$ is a polynomial, this yields a contradiction. We conclude we must have had $b = \infty$, i.e. $b_{m} \rightarrow \infty$, and $$A|b_m|^{k} \sim |P_n'(b_m)| = |q(b_m)|^{m-1} |m b_m q'(b_m) + q(b_m)|$$ as $m \rightarrow \infty$ where $A$ is the modulus of the leading coefficient of $P_n'$ and $k = \deg(P_n) - 1$. Note that $A$ and $k$ are \textbf{fixed}. Now we have that $$|q(z)| \geq |z|^{\frac{d}{2}}$$ and $$|m z q'(z) + q(z)| \geq C_1 m |z|^d$$ for all $|z| \geq R$, where $C_1>0$ is a fixed positive constant, $d = \deg(q)$ and $R$ is chosen sufficiently large. The second inequality follows from the fact that $$mzq'(z) + q(z) = (md+1)a_d z^d + (m(d-1) + 1)a_{d-1} z^{d-1} + \ldots + (m + 1)a_1 z + a_0$$ where $q(z) = a_dz^d + \ldots + a_1z + a_0$, so that $$|mzq'(z) + q(z)| \geq m \frac{d|a_d|}{2} |z|^d$$ whenever $|z|$ is sufficiently large (independently of $m$), hence we may take $C_1 = \frac{d |a_d|}{2}$. Thus $$|P_n'(b_m)| =  |q(b_m)|^{m-1} |m b_m q'(b_m) + q(b_m)| \geq C_2 |b_m|^{\frac{d(m+1)}{2}} $$ for some fixed positive constant $C_2 > 0$ as $m \rightarrow \infty$ and we conclude that $$A |b_m|^{k} \gtrsim  C_2 |b_m|^{\frac{d(m+1)}{2}}$$ as $m \rightarrow \infty$, which yields a contradiction since $A,k, C_2$ are fixed positive constants, $d \geq 1$, and $b_m \rightarrow \infty$ as $m \rightarrow \infty$. We have thus proved that all critical points of $H_{m,n}$ lie in the annular neighbourhood bounded by the analytic Jordan curves $\Gamma_{r}^{-}$ and $\Gamma_{r}^{+}$ so that in particular, all critical points of $H_{m,n}$ lie in (a compact subset of) the annular neighbourhood bounded by $W( (1 - \epsilon) \partial \mathbb{D})$ and $W( (1 + \epsilon) \partial \mathbb{D})$ given $m$ is chosen sufficiently large.

Similar considerations show that at a critical point $b \in Z(H_{m,n}')$ of $H_{m,n}$, we have $$q^{m}(b)=\frac{- P_n'(b)q(b)}{mbq'(b) + q(b)} = O(\frac{1}{m})$$ as $m \rightarrow \infty$. In particular, $$|\frac{H_{m,n}(b)}{b}| \ge |\frac{P_n(b)}{b}| - O(\frac{1}{m}) > A+c - O(\frac{1}{m})$$ at every critical point $b \in Z(H_{m,n}')$, for all $m$ sufficiently large. 

Rescaling $H_{m,n}$ via $$H_{m,n} \rightarrow \frac{H_{m,n}}{H_{m,n}'(0)}$$ where $$H_{m,n}'(0)=1+q(0)^m = 1 + o(1)$$ and  $$o(1) \rightarrow_{m \rightarrow \infty} 0$$ so that the resulting polynomial has derivative at the origin equal to $1$, and suppressing renaming the resulting rescaled polynomial, our resulting (rescaled) polynomial $H_{m,n}$ has $S(H_{m,n})> A$ for $m$ sufficiently large, and this resolves the second issue. 

We have thus shown that for every admissible pair $(W,B)$ and every positive real number $A < \frac{1}{B(0)}$, there is a polynomial $P \in \mathcal{P}$ satisfying $S(P) > A$. Thus $\Bee\le\M$. The inequality $\Bee_d\le\M$ follows from $\Bee_{d} \leq \Bee$ and this concludes the proof.
\end{proof}

Note that in the above proof, we only used that $\tilde{P}$ is holomorphic in a neighbourhood of the analytic Jordan domain $W(\overline{\mathbb{D}})$, and that $\tilde{P}(0) = 0$, $\tilde{P}'(0) = 1$ and $\tilde{P}'$ is nowhere vanishing in a neighbourhood of the analytic Jordan domain $W(\overline{\mathbb{D}})$, together with the inequality $|\tilde P(w)/w|>A$ on $W(\partial\mathbb D)$. By the same argument we thereby obtain the following result.

\begin{proposition}\label{prop:completion}
Let $\Omega$ be a bounded analytic Jordan domain containing $0$.
Suppose $f$ is holomorphic on a neighbourhood of $\cl\Omega$,
$$
f(0)=0,\qquad f'(0)=1,\qquad f'\ne0\text{ on }\cl\Omega,
$$
and, for some $A>0$,
$$
\left|\frac{f(w)}{w}\right|>A
\qquad(w\in\partial\Omega).
$$
Then there is a polynomial $P\in\cP$ with $S(P)>A$. More precisely, one can choose polynomials $p,q$, independent
of $n$, such that
$$
P_n(w)=\frac{p(w)+wq(w)^n}{1+q(0)^n}
$$
has $S(P_n)>A$ for all sufficiently large $n$.
If $k=\deg p$ and $m=\deg q$, then
$\deg P_n=mn+1$ whenever $mn+1>k$.
\end{proposition}

\subsection{$\mathcal{A}=\M$}\label{sec:AM}

Proposition~\ref{prop:completion} suggests one may consider another reformulation of Smale's mean value constant $\M$ which we detail in this section.

In this section, a pair $(f, \Omega)$ will be called \textbf{admissible} if $\Omega$ is a bounded analytic Jordan domain containing the origin in its interior, and $f$ is a holomorphic function, locally univalent on a neighbourhood of $\overline{\Omega}$, fixing the origin, with unit derivative at the origin.
Set
$$
\cA:=\sup\{\inf_{z \in \partial \Omega} |\frac{f(z)}{z}|:(f, \Omega) \text{ is admissible}\}.
$$
Since the pair $(z, \mathbb{D})$ is clearly admissible, $\mathcal{A} \geq 1$.
Proposition~\ref{prop:completion} is exactly the statement that for each admissible pair $(f, \Omega)$ and each $0<A<\inf_{z\in\partial\Omega}|f(z)/z|$, there is a corresponding polynomial $P \in \mathcal{P}$ with $S(P) > A$ and hence $$\M\ \ge\ \inf_{z \in \partial\Omega} \Bigl|\frac{f(z)}{z}\Bigr|.$$ This motivates the next theorem. 

\begin{theorem}\label{thm:AM}
$$\mathcal{A} = \M$$
\end{theorem}

\begin{proof}
Proposition \ref{prop:completion} immediately implies that $\mathcal{A} \leq \M$. To prove $\mathcal{A} \geq \M$ we may assume without loss of generality that $\M > 1$, for if $\M \leq 1$, then $\M \leq 1 \leq \mathcal{A}$ so there is nothing to prove. Let $P \in \mathcal{P}$ be any polynomial for which $S(P) > 1$. Then one may set $f = P$ and let $\Omega$ be the component of the sublevel set $$\{ |\frac{P(z)}{z}| < t \}$$ containing the origin, where $1< t < S(P)$ is any value which is not equal to a modulus of the (finitely many) critical values of the polynomial $$Q(z) \coloneq \frac{P(z)}{z}.$$ Then $Q'\neq0$ on the level set $\{|Q|=t\}$, which is therefore a finite union of disjoint analytic Jordan curves. As in the proof of Proposition~\ref{prop:poly-to-pair}, $\Omega$ is bounded and simply connected, and $\partial\Omega\subset\{|Q|=t\}$.So $\Omega$ is an analytic Jordan domain containing the origin and $|Q|\le t<S(P)$ on $\cl\Omega$. On the other hand, $|Q(c)|\ge S(P)$ at every critical point $c$ of $P$ and therefore $P'\neq0$ on the compact set $\cl\Omega$ and, by continuity, on a neighbourhood of $\cl\Omega$. It follows that $(f,\Omega)$ is an admissible pair with $$\inf_{z \in \partial \Omega} |\frac{P(z)}{z}| = \inf_{z \in \partial \Omega} |Q(z)| = t$$

Letting $t \uparrow S(P)$ (through values avoiding the finitely many excluded moduli) , we find that $\mathcal{A} \geq S(P)$, and since $P \in \mathcal{P}$ was arbitrary with $S(P) > 1$, taking the supremum of $S(P)$ over $P \in \mathcal{P}$ with $S(P) > 1$ gives us $$\mathcal{A} \geq \M$$
\end{proof}

\begin{remark}
Theorem~\ref{thm:AM} says that Smale's extremal problem does not distinguish between polynomials and arbitrary locally univalent holomorphic functions. In spirit this is an algebraization result of Chow's Lemma type: the analytic category of admissible pairs may be replaced by the algebraic category of polynomials
\end{remark}

Combined with Proposition~\ref{prop:BM} and Proposition~\ref{prop:poly-to-pair}, this gives
$$
\Bee=\M=\cA.
$$
%%%%%%%%%%%%%%

\begin{corollary}\label{cor:infinitely}
For every $A<\M$ there are infinitely many $d$ with $\M_d>A$. We also have $\M_d>1.1018$ for all large $d$.
\end{corollary}

\begin{proof}
Since $\cA=\M>A$ there is an admissible $(f,\Omega)$ with $\inf_{\partial\Omega}|f(w)/w|>A$. Proposition~\ref{prop:completion} produces polynomials $P_n\in\cP$ with $S(P_n)>A$ and $\deg P_n=mn+1\to\infty$ (for all sufficiently large $n$, where $m=\deg q\ge1$). By Theorem~\ref{thm:lower} below, $\M>1.1018586788\ldots$. So we can take $A=1.1018586788\ldots$ and hence $\M_{mn+1}\ge 1.1018586788\ldots$ for all large $m$. Since $\M_{d+1}\ge \M_{d}$ \cite{Crane1}, we  have $\M_d> 1.1018$  for all large $d$.
\end{proof}
%%%%%%%%%%%%%%%%

\section{Smale's Problems 1A and 1B}\label{sec:smale1A}

\subsection{Smale's constants $\K_{d,k}$, $\K_k$ and $\K$}\label{sec:Kdef}
In Smale's studies of the efficiency and complexity of Newton's method \cite{Smale}, 
the analysis of one step $w\mapsto w-hP(w)/P'(w)$ of the modified Newton iteration in \cite[Part~II, \S\S1--3]{Smale} is carried out in the target space of $P$, and it rests on estimates of the critical values of $P$ in terms of the Taylor coefficients of $P$ at the current point. Translating the variable and the value, one may assume that the current point is the origin and that $P(0)=0$. Smale then introduced and estimated the following quantities. For $P\in\cP$ and an integer $k\ge2$ let
\begin{equation}\label{eq:Sk}
S_k(P):=\min_{P'(c)=0}\Bigl|\frac{P^{(k)}(0)}{k!}\Bigr|^{\frac1{k-1}}|P(c)|
=|\alpha_k|^{\frac1{k-1}}\min_{P'(c)=0}|P(c)|,\qquad P(w)=w+\sum_{j\ge2}\alpha_jw^j,
\end{equation}
so that $S_k(P)=0$ when $k>\deg P$, and put
$$
\K_{d,k}:=\sup_{P\in\cP_d}S_k(P),\qquad
\K_k:=\sup_{P\in\cP}S_k(P)=\sup_{d\ge2}\K_{d,k},\qquad
\K:=\sup_{k\ge2}\K_k .
$$
For a polynomial $P$ with $P(0)=0$ and $P'(0)\ne0$ which is not normalized, the corresponding quantity is
\begin{equation}\label{eq:Skgen}
\min_{P'(c)=0}\Bigl|\frac{P^{(k)}(0)}{k!\,P'(0)^k}\Bigr|^{\frac1{k-1}}|P(c)| ,
\end{equation}
which is invariant under $P(w)\mapsto\alpha P(\beta w)$ ($\alpha\beta\ne0$) and coincides with $S_k(P)$ for $P\in\cP$. Applying \eqref{eq:Skgen} to $w\mapsto P(a+w)-P(a)$, one sees that $\K$ is the least constant with the following property: for every nonlinear polynomial $P$ and every non-critical point $a$ of $P$ there is a critical point $c$ of $P$ such that for every $k\ge2$
\begin{equation}\label{eq:SmaleThm1}
\Bigl|\frac{P^{(k)}(a)}{k!\,P'(a)}\Bigr|^{\frac1{k-1}}|P(a)-P(c)|\le\K\,|P'(a)| .
\end{equation}
In the same way $\K_{d,k}$ is the least constant in \eqref{eq:SmaleThm1} for a fixed $k$ and polynomials of degree $d$. For $k=2$ one has $\alpha_2=-\frac12\sum_i1/c_i$, the sum running over the critical points $c_i$ of $P$ counted with multiplicity, so that
$$
S_2(P)=\frac12\,\min_{j}\Bigl|\sum_{i}\frac{P(c_j)}{c_i}\Bigr| ,
$$
a quantity which resembles, but is independent of, the quantity $S(P)=\min_j|P(c_j)/c_j|$ of Section~\ref{sec:smale}.

Let
$$
C_k:=\frac{(2k)!}{k!\,(k+1)!}=\frac1{k+1}\binom{2k}{k}\qquad(k\ge1)
$$
denote the $k$-th Catalan number, so that $C_1=1$, $C_2=2$, $C_3=5$, $C_4=14$, $C_5=42$ and $C_k=\frac{2(2k-1)}{k+1}C_{k-1}<4C_{k-1}$. The inverse of the Koebe function $\kappa(z)=z/(1-z)^2$ has the expansion
\begin{equation}\label{eq:invKoebe}
\kappa^{-1}(w)=\sum_{k\ge1}(-1)^{k-1}C_k\,w^k\qquad(|w|<1/4),
\end{equation}
and L\"owner's theorem \cite{Loewner1923} (see also \cite{Duren}) asserts that if $W\in\cS$ and $W^{-1}(w)=w+\sum_{k\ge2}A_kw^k$ is the expansion of the inverse function near the origin, then $|A_k|\le C_k$ for every $k\ge2$, with equality for the rotations of the Koebe function. Using this theorem Smale proved \cite[Theorem~1, Remarks~1--2 and Lemma~2]{Smale} that
\begin{equation}\label{eq:SmaleK}
\K_k\le C_k^{\frac1{k-1}}<4\qquad(k\ge2),\qquad\text{hence}\qquad\K\le4 .
\end{equation}
He also observed that $\K\ge1$, since $S_d(w-w^d/d)=(d-1)d^{-d/(d-1)}\to1$ as $d\to\infty$, while the polynomials $P(w)=(1-(1-w)^d)/d\in\cP_d$, whose only critical point is $w=1$, give $S_2(P)=\frac{d-1}{2d}$. The constant $\K$ enters his convergence criteria for Newton's method: if $\rho_P:=\min_{P'(c)=0}|P(c)|>0$ and $|P(w)|<\rho_P/(3\K+1)$, then Newton's iteration started at $w$ converges to a zero of $P$ and $|P(w')|<\frac12|P(w)|$ for $w'=w-P(w)/P'(w)$ \cite[Part~II, Corollaries~A and B]{Smale}. In the list of open problems in \cite[Part~III]{Smale} Smale asked (in the notation above):
\begin{itemize}
\item[]\emph{Problem 1A.} Reduce the constant $4$ in $\K\le4$. Is it possible that $\K=1$?
\item[]\emph{Problem 1B.} For $k=2$ inequality \eqref{eq:SmaleK} gives $\K_{d,2}\le2$. Can $2$ be replaced by $\frac{d-1}{2d}$, i.e.\ is $\K_{d,2}=\frac{d-1}{2d}$?
\end{itemize}
Problems 1C and 1D of \cite{Smale} ask whether $(1-(1-w)^d)/d$ is a (local) maximizer of $S_2$ on $\cP_d$, and Problem 1E is the mean value conjecture discussed in Section~\ref{sec:smale}.

For the top coefficient $k=d$ the sharp value of $\K_{d,k}$ is known. Dubinin \cite{Dubinin2006}, using the method of dissymmetrization, proved that every polynomial $P(w)=w^d+\dots+c_1w$ satisfies
$$
\min_{P'(c)=0}|P(c)|\le(d-1)\Bigl|\frac{c_1}{d}\Bigr|^{\frac d{d-1}},
$$
with equality for $P(w)=w^d+c_1w$; in the notation above this reads
\begin{equation}\label{eq:Dubinin}
\K_{d,d}=\frac{d-1}{d^{d/(d-1)}}\qquad(d\ge2),
\end{equation}
so Smale's example $w-w^d/d$ is extremal for $k=d$, and $\K_{d,d}<1$ for every $d$. Our result below concerns the opposite regime of a fixed $k$ and unbounded degree. 

\subsection{The sharp constants $\K_k$ and $\K$}

\begin{theorem}\label{thm:Kk}
For every integer $k\ge2$,
$$
\K_k=\sup_{P\in\cP}S_k(P)=C_k^{\frac1{k-1}},
$$
and for every $A<C_k^{1/(k-1)}$ there are polynomials $P\in\cP$ of arbitrarily large degree with $S_k(P)>A$. Consequently
$$
\K_2=2\qquad\text{and}\qquad\K=\sup_{k\ge2}C_k^{\frac1{k-1}}=\lim_{k\to\infty}C_k^{\frac1{k-1}}=4 .
$$
In particular the constant $4$ in Smale's Problem 1A cannot be reduced, and $\K_{d,2}$ is \textbf{arbitrarily close to} $2$ for sufficiently large $d$, so the answer to Problem 1B is also negative.
\end{theorem}

The upper bound in Theorem~\ref{thm:Kk} is Smale's inequality \eqref{eq:SmaleK}; for the reader's convenience we include its short proof. The new content is the lower bound, which is obtained from the following variant of Proposition~\ref{prop:completion}, in which $|f(w)/w|$ is replaced by $|f(w)|$ and the $k$-th Taylor coefficient at the origin is kept under control.

\begin{proposition}\label{prop:completionK}
Let $\Omega$ be a bounded analytic Jordan domain containing $0$, and let $f$ be holomorphic on a neighbourhood of $\cl\Omega$ with
$$
f(0)=0,\qquad f'(0)=1,\qquad f'\ne0\ \text{on}\ \cl\Omega .
$$
Let $k\ge2$ and suppose that
$$
0<A<\Bigl|\frac{f^{(k)}(0)}{k!}\Bigr|^{\frac1{k-1}}\min_{w\in\partial\Omega}|f(w)| .
$$
Then there is a polynomial $P\in\cP$ with $S_k(P)>A$. More precisely, there are polynomials $p$, $q$ independent of $m$ such that
$$
Q_m(w)=\frac{p(w)+wq(w)^m}{1+q(0)^m}
$$
satisfies $S_k(Q_m)>A$ for all sufficiently large $m$, and $\deg Q_m=m\deg q+1$ for large $m$.
\end{proposition}

\begin{proof}
This follows in essentially the same way as the proof of Proposition~\ref{prop:completion}. In the notation of the proof of Proposition~\ref{prop:completion}, the only points to note are the following.

\begin{itemize}
    \item $$q^{m}(b) = O(\frac{1}{m})$$ at all critical points $b \in Z(H_{m,n}')$ implies $$bq^{m}(b) = O(\frac{1}{m})$$ at all critical points $b \in Z(H_{m,n}')$ since all critical points $b \in Z(H_{m,n}')$ lie in the annular neighbourhood bounded by $\Gamma_{r}^{-}$ and $\Gamma_{r}^{+}$ and thus satisfy $|b| \leq C$, for some fixed finite constant $0<C < \infty$. 
    \item For any fixed radius $r > 0$, chosen sufficiently small so that $\Delta(0;2r) \subset \mathrm{int}(\Gamma_{r}^{-})$ holds, $$H_{m,n}$$ converges uniformly on $\overline{\Delta(0;r)}$ to $$P_{n}$$ as $m \rightarrow \infty$
    \item Uniform convergence of $H_{m,n}$ to $P_n$ on the disk $\overline{\Delta(0;r)}$ implies $H_{m,n}^{(k)}(0) \rightarrow_{m \rightarrow \infty} P_{n}^{(k)}(0)$ by the Cauchy integral formula (note that $k$ is \textbf{fixed})
\end{itemize}

Keeping the above three points in mind, the desired conclusion follows in exactly the same way as in the proof of Proposition~\ref{prop:completion}.

\end{proof}

\begin{proof}[Proof of Theorem~\ref{thm:Kk}] In \cite{Smale} Smale made use of L\"owner's theorem \cite{Loewner1923} to get 
$$
S_k(P)\le C_k^{\frac1{k-1}} .
$$
Taking the supremum over $P$ gives $\K_k\le C_k^{1/(k-1)}$. Moreover $C_1=1$ and $C_k<4C_{k-1}$ give $C_k<4^{k-1}$, i.e.\ $C_k^{1/(k-1)}<4$, so $\K\le4$.

To get the lower bound, fix $\epsilon\in(0,1)$, let $\Omega_\epsilon:=\kappa(D(0,1-\epsilon))$ and let $f:=\kappa^{-1}$, which is holomorphic on the neighbourhood $\kappa(D(0,1-\epsilon/2))$ of $\cl{\Omega_\epsilon}$. Since $\kappa$ is univalent and holomorphic on a neighbourhood of $\cl{D(0,1-\epsilon)}$, $\Omega_\epsilon$ is a bounded analytic Jordan domain containing $0$. Furthermore $f(0)=0$, $f'(0)=1$, and $f'=1/(\kappa'\circ f)\ne0$ because $\kappa'(z)=(1+z)/(1-z)^3$ has no pole in $\D$. On $\partial\Omega_\epsilon=\kappa((1-\epsilon)\partial\D)$ we have $|f|=1-\epsilon$, and by \eqref{eq:invKoebe} $|f^{(k)}(0)/k!|=C_k$. Hence
$$
\Bigl|\frac{f^{(k)}(0)}{k!}\Bigr|^{\frac1{k-1}}\min_{\partial\Omega_\epsilon}|f|=(1-\epsilon)\,C_k^{\frac1{k-1}},
$$
and Proposition~\ref{prop:completionK} (with $A=(1-2\epsilon)C_k^{1/(k-1)}$) yields polynomials $P\in\cP$ of arbitrarily large degree with $S_k(P)>(1-2\epsilon)C_k^{1/(k-1)}$. Letting $\epsilon\downarrow0$ gives $\K_k\ge C_k^{1/(k-1)}$ together with the assertion about large degrees.

 By Stirling's formula $C_k\sim4^k/(\sqrt\pi\,k^{3/2})$, so $C_k^{1/(k-1)}\to4$, and since $C_k^{1/(k-1)}<4$ for all $k$ we get $\K=\sup_k\K_k=4$. Finally $\K_2=\sup_d\K_{d,2}=2$, and the constants $\mathcal{K}_{d,k}$ are increasing in the degree $d$ (see Remark (a) below), so that $\K_{d,2}$ can be made arbitrarily close to $2$ for sufficiently large $d$.
\end{proof}

\begin{remark}\label{remove}
We record here the fact that $\K_{d,k}$ is increasing in the degree $d$, and a consequence of the resolution of Problem 1A for Newton's method.

\smallskip\noindent
\emph{(a) Monotonicity in the degree.}
For $2\le k\le d$ one has $\K_{d+1,k}\ge\K_{d,k}$. Indeed, for $P\in\cP_d$ and
$\eta\ne0$ let $P_\eta:=P+\eta w^{d+1}\in\cP_{d+1}$, so that
$\alpha_j(P_\eta)=\alpha_j(P)$ for $j\le d$. As $\eta\to0$, by Hurwitz's theorem $d-1$
critical points of $P_\eta$ converge to those of $P$, and the remaining one $c_\eta$
satisfies $P'(c_\eta)=-(d+1)\eta c_\eta^{d}$, hence $\eta c_\eta\to-d\alpha_d/(d+1)$
and $P_\eta(c_\eta)=c_\eta^d(\alpha_d+\eta c_\eta)+O(c_\eta^{d-1})\to\infty$. Thus
$\rho_{P_\eta}\to\rho_P$ and $S_k(P_\eta)\to S_k(P)$. Consequently
$\K_{d,k}\uparrow\K_k$ as $d\to\infty$.

\smallskip\noindent
\emph{(b) Consequences for Newton's method and for $\K_{d,k}$.}
Because $\K=4$, the constant $4$ in Smale's convergence criteria for Newton's method
\cite[Part~II]{Smale} cannot be improved as long as these criteria are derived from
\eqref{eq:SmaleThm1}. On the other hand, \eqref{eq:Dubinin} shows that the
degree-restricted constants $\K_{d,k}$ can be much smaller for particular $k$; their
exact values for $2\le k<d$ and the corresponding extremal polynomials
(cf.\ Problems 1C and 1D of \cite{Smale}) remain to be determined.
\end{remark}

\section{The degree-one lower bound}\label{sec:quartic}

We will prove the following result by constructing an explicit admissible degree one univalent Blaschke pair $(W,B)$.
\begin{theorem}\label{thm:lower}
$$
\Bee_1\ge \sqrt{\frac{8081}{6656}}=1.1018586788\ldots.
$$
\end{theorem}

There are mainly two ways to construct admissible pairs of degree one. Either to begin with i) a known univalent function or ii) a known locally univalent holomorphic function on the unit disk. One useful and easily proved fact is that if $W\in\cS$ is starlike (with respect to zero) and $B$ is a nonconstant finite Blaschke product with $B(0)\neq 0$, then $(WB)'$ has a zero in $\D$ (see Section~\ref{sec:cr} for the proof). So if one begins with a univalent function (e.g. Koebe function), one needs to modify it so that it is non-starlike (with respect to zero) and then try to find a suitable degree one Blaschke product $B$ so that $(WB)'\neq 0$ on $\D$ (which is a very restrictive condition). We were able to produce a number of admissible pairs this way. However, the one that gives us the largest $1/B(0)$ follows another approach. The idea is to start with a locally univalent function on $\D$, say $Q$ and then consider a Blaschke factor $B_{\alpha}=\frac{z-\alpha}{1-\cl{\alpha}z}$ whose only zero $\alpha$ is also a zero of $Q$. Then try to prove that $Q/B_{\alpha} \in \cS$ and hence it can be our $W$. This approach is also not easy to implement. However, when we start with a polynomial $Q$ with $Q'\neq 0$ on $\D$, then $Q/B_{\alpha}$ will also be a polynomial and we can then apply a test of Dieudonn\'e to show that $Q/B_{\alpha}$ is univalent. 

\begin{proof}
To begin with, we consider a polynomial $Q$ with at least two distinct zeros in $\D$, say $a$ and $b$ and  $Q'$ is non-zero throughout $\D$, and with
\begin{equation}\label{eq:H}
H(z):=\frac{Q(z)}{B_a(z)}=\frac{(1-\cl a\,z)\,Q(z)}{z-a}
\end{equation}
univalent in $\D$.
Then $H(b)=0$ and
\begin{equation}\label{eq:eta}
\eta:=B_a(b)=\frac{b-a}{1-\cl a\,b}
\end{equation}
has modulus equal to the pseudo-hyperbolic distance between $a$ and $b$. Then the following lemma produce an admissible pair.

\begin{lemma}(Change of base) \label{lem:base}
Under the hypotheses above, set $\phi(z)=(b+z)/(1+\cl b\,z)$ and $N=(1-|b|^2)H'(b)$.
Then
\begin{equation}\label{eq:pair}
W(z):=\frac{H(\phi(z))}{N},
\qquad
B(z):=\frac{|\eta|}{\eta}B_a(\phi(z))
\end{equation}
is an admissible pair of degree one, and $B(0)=|\eta|$.
\end{lemma}

\begin{proof}
The map $\phi$ is an automorphism of $\D$ with $\phi(0)=b$, and $\phi'(0)=1-|b|^2$.
Since $H$ is univalent and $H(b)=0$, the function $W$ is univalent, $W(0)=0$, and
$$
W'(0)=\frac{H'(b)\,\phi'(0)}{N}=1
$$
and therefore $W\in\cS$. The factor $B$ is a degree-one Blaschke product with $B(0)=\frac{|\eta|}{\eta}B_a(b)=|\eta|>0$. By construction, $WB$ is a non-zero constant multiple of $Q\circ\phi$ and hence $(WB)'(z) \neq 0$ on $\D$ as $\phi$ is an automorphism of $\D$
\end{proof}
By Lemma \ref{lem:base} and the identity
$|1-\cl a\,b|^2=|a-b|^2+(1-|a|^2)(1-|b|^2)$, to maximize $1/B(0)$ is the same as to maximize the following elementary expansion
\begin{equation}\label{eq:pseudo}
\frac1{|\eta|^2}
=
1+\frac{(1-|a|^2)(1-|b|^2)}{|a-b|^2}
\end{equation}
From (\ref{eq:pseudo}), we want to make $a$ and $b$ close but at the same time cannot be too close as then $Q'$ and hence $(WB)'$ may have a zero in $\D$.

The maybe simplest way to begin with the construction of $Q$ is to take $Q'(z)=d(z-1)^{d-1}$ where $d\ge 2$ so that $Q' \neq 0$ on $\D$. Then we can write $Q(z)=(z-1)^d-u^d$. It turns out that when $d=4$, we can find a suitable $u$ that produces an admissible pair with $1/B(0)>1.1018$. To find the suitable $u$, we let $u=-x-iy$ with $x,y>0$. Then the four zeros of $(z-1)^4-u^4$ are 
\begin{align*}
1+u&=1-x-iy,\\
1-iu&=1-y+ix,\\
1-u&=1+x+iy,\\
1+iu&=1+y-ix.
\end{align*}

The last two have real part greater than $1$, so they stay outside $\overline{\mathbb D}$. Let $a=1-x-iy$ and $b=1-y+ix$ and $s=x^2+y^2$. Then both $a$ and $b$ lie in $\D$ if and only if
\begin{equation}\label{eq:indisk}
(x-1)^2+y^2<1
\qquad\text{and}\qquad
x^2+(y-1)^2<1,
\end{equation}
equivalently $s<2x$ and $s<2y$. Thus this search region $\cR$ is the intersection of two open unit disks, centred at $(1,0)$ and at $(0,1)$.
Whenever the quotient $H$ 
is univalent in $\D$, Lemma~\ref{lem:base} produces an admissible pair, and \eqref{eq:pseudo} gives
\begin{equation}\label{eq:J}
J(x,y):=\left(\frac1{B(0)}\right)^2
=
1+\frac{2xy}{s}-(x+y)+\frac s2.
\end{equation}
We search rational points $(x,y) \in \cR$ so that $J(x,y)$ is as large as possible subject to the degree $4$ polynomial $H$ (which depends on $x,y$) is univalent in $\D$. We will use the Dieudonn\'e's criterion combined with the Schur--Cohn test to help us check the univalence of $H$.

\begin{lemma}[Dieudonn\'e]\label{lem:dieu}
A polynomial $H(z)=\sum_{k=1}^n a_k z^k$ is univalent in $\D$ if and only if, for every $c\in[-1,1]$, its Dieudonn\'e polynomial
$$
P_c(w):=\sum_{k=1}^n a_k\,U_{k-1}(c)\,w^{k-1}
$$
has no zero in $\D$. Here $U_n$ is the Chebyshev polynomial of the second kind. 
\end{lemma}

The classical statement uses $c=\cos\theta$ and forbids zeros in the open disk only \cite{Dieudonne}.  A zero of some $P_c$ on $\T$ need not destroy univalence in $\D$. The criterion used here is the sufficient closed-disk condition, which is what the Schur--Cohn test detects.

\begin{lemma}[Schur--Cohn test]\label{lem:schur}
Let $p(w)=\sum_{k=0}^n\alpha_k w^k$ with $\alpha_n\neq0$, and let
$p^*(w)=\sum_{k=0}^n\cl{\alpha_{n-k}}\,w^k$.
The transform
$$
(Tp)(w)=\cl{\alpha_0}\,p(w)-\alpha_n\,p^*(w)
$$
has degree at most $n-1$, and its constant term is $\delta=|\alpha_0|^2-|\alpha_n|^2$.
If each of the $n$ successive transforms has positive constant term and drops degree by exactly one, then $p$ has no zero in the closed disk.
\end{lemma}

The coefficient of $w^{n-1}$ in $Tp$ is $\cl{\alpha_0}\,\alpha_{n-1}-\alpha_n\,\cl{\alpha_1}$. For the  quartic polynomial $H$, the leading coefficient of its Dieudonn\'e polynomial $P_c$ is a non-zero multiple of $U_3(c)=4c(2c^2-1)$. So $P_c$ will not be cubic if $c=0$ and $c=\pm1/\sqrt2$.

Assume $c$ is not one of these three values. To apply Lemma \ref{lem:schur} repeatedly to the cubic $P_c$ to make sure that it has no zeros in $\D$ and hence $H$ is univalent in $\D$ by Lemma \ref{lem:dieu}, one needs to check three constant terms and two leading coefficients for $TP_c$ and $T(TP_c)$ respectively.
All five are even polynomials in $c$ (because $U_k(-c)=(-1)^kU_k(c)$ gives $P_{-c}(w)=P_c(-w)$, and the Schur--Cohn data are invariant under $w\mapsto-w$).
After a positive rational factor and the integer content are removed, they are the primitive polynomials $D_1,D_2,D_3,L_1$ and $c^2M(c)$ of the appendix. Grok-4.6 was then used  to search for rational $(x,y)\in \cR$ so that these five polynomials are positive and $J(x,y)$ is as large as possible. It turns out that for $x=9/16$ and $y=3/8$, 
$$\frac{1}{B(0)}=\sqrt{J(\frac{9}{16},\frac{3}{8})}=\sqrt{\frac{8081}{6656}}$$
and we also have $H(z)=a_0+a_1z+a_2z^2+a_3z^3+a_4z^4$
with
\begin{align*}
a_0&=-\frac{6877}{4096}-\frac{2253}{2048}\,i,&
a_1&=\frac{302959}{65536}+\frac{37401}{16384}\,i,\\
a_2&=-\frac{20499}{4096}-\frac{5121}{2048}\,i,&
a_3&=\frac{619}{256}+\frac32\,i,&
a_4&=-\frac7{16}-\frac38\,i.
\end{align*}
To apply Lemma \ref{lem:dieu}, we need to  consider $H-H(0)$ instead of $H$ and this is fine as univalence is translation-invariant.
For this $H-H(0)$'s Dieudonn\'e polynomial $P_c$, we have the following two lemmas.

\begin{lemma}\label{lem:deltas}
When $x=9/16$ and $y=3/8$, we have for every $c\in[-1,1]$ other than $0$ and $\pm1/\sqrt2$, the Dieudonn\'e polynomial $P_c$ has no zero in $\cl\D$.
\end{lemma}

\begin{proof}
By Sturm's theorem (see Section 10.5 of \cite{RS}), one finds no zero of $D_1$, $D_2$ or $D_3$ in $(-1,1)$.
Each is positive at $c=0$ and at $c=\pm1$, hence positive on $[-1,1]$. (For $D_1$ this can be checked by hand: $D_1(c)=114165512497-22817013760\,c^2(1-2c^2)^2$, and $\max_{0\le\tau\le1}\tau(1-2\tau)^2=1$.)
For every $c\in[-1,1]$ other than $0$ and $\pm1/\sqrt2$, these values $D_i$ are the three Schur constants of the $P_c$, so those constants are positive there.

The same count shows that $L_1$ has no root in $(-1,1)$ and is positive at the endpoints, so the first transform drops degree by exactly one everywhere on the interval.
The squared modulus of the next leading coefficient of $T(TP_c)$ is a positive rational multiple of $c^2 M(c)$. Sturm's theorem finds exactly one distinct root of $c^2M(c)$ in $(-1,1)$.
As $c^2M(c)=0$ at $c=0$, this implies that this root is $c=0$ and $M$ has no root in $(-1,1)$.
Since $M(0)$ and $M(\pm1)$ are positive, $M>0$ on $[-1,1]$, and the second leading coefficient $c^2M(c)$ is positive on the closed interval as $c \ne 0$. So both transforms drop degree by exactly one, and all three constant terms are positive.
Lemma~\ref{lem:schur} gives the claim.
\end{proof}

It remains to prove the following  result to complete the proof of Theorem \ref{thm:lower}.

\begin{lemma}\label{lem:drop}
When $x=9/16$ and $y=3/8$, the Dieudonn\'e polynomial $P_c$
at $c=0$ and at $c=\pm1/\sqrt2$ have no zero in $\cl\D$.
\end{lemma}

\begin{proof}
At $c=0$ one has $U_1(0)=U_3(0)=0$ and $U_2(0)=-1$, so $P_0(w)=a_1-a_3w^2$.
The roots satisfy $|w|^2=|a_1|/|a_3|$, and since the coefficient of $w$ in $P_0$ vanishes, we have  
$$
TP_0(w)=|a_1|^2-|a_3|^2=\frac{79390996785}{4294967296}>0.
$$
Hence, both roots lie outside the closed disk.

At $c=\pm1/\sqrt2$ one has $U_3(c)=0$ and $U_2(c)=1$, so the $P_c$ is the quadratic
$$
a_1\pm a_2\sqrt2\,w+a_3w^2.
$$
The constant term of $TP_c$ is the same positive number $|a_1|^2-|a_3|^2$, and its leading coefficient
$$
\pm\sqrt2\,(\cl{a_1}a_2-a_3\cl{a_2})=\pm\sqrt2\;\frac{-3487417605+355971726\,i}{2^{28}}
$$
is nonzero, so $TP_c$ has degree exactly one.
The constant term of $T(TP_c)$ is 
$$
\frac{11066090771698524513}{18446744073709551616}>0.
$$
for both $c$. By Lemma~\ref{lem:schur}, $P_c$ has no zero in $\cl\D$ for these three values of $c$.
\end{proof}

The two lemmas say that every Dieudonn\'e polynomial $P_c$, $c\in[-1,1]$, is zero-free in $\cl\D$.
Dieudonn\'e's criterion gives univalence of $H$, which completes the proof of Theorem~\ref{thm:lower}.
\end{proof}

\begin{remark}
    If we take $x=14/25$ and $y=3/8$, then we can improve the lower bound of $\Bee_1$ to $\Bee_1> 1.1030$. It is likely that this value can be further improved a bit but we did not try hard to do so as we do not see how to improve the value to get really close to the upper bound $\Bee_1 < 1.1833$ given in the next section.
\end{remark}

\section{Upper bounds of $\Bee$, $\Bee_d$ and $\Bee_1$} \label{sec:four}

In this section, we will give upper bounds of $\Bee$, $\Bee_d$ and $\Bee_1$. The precise values of these constant are still open. 
\subsection{The bounds $\Bee\le4$ and $\Bee_d \le 4-C/d^2$}

For any $W\in\cS$, let $m(W):=\dist\bigl(0,\partial W(\D)\bigr)$. By (T1), $m(W)\ge1/4$.

\begin{theorem}\label{thm:upper}
Let $(W,B)$ be admissible. Then $B(0)\ge m(W)\ge\frac14$; hence $\Bee\le4$, and by Theorem~\ref{thm:main} this recovers Smale's theorem $\M\le4$. If moreover $B$ is nonconstant and $a\in\D$ is a zero of $B$, then
$$
B(0)\ \ge\ \frac{(1+|a|)^2}{4|a|}\,m(W)\ \ge\ \frac{(1+|a|)^2}{16|a|}>\frac14 ,
$$
and if $\deg B=d\ge1$ then
$$\frac1{B(0)} < 4-\frac{529}{576}\frac{(\log 2)^2}{d^2}.$$
Consequently $\Bee_d\le 4-\frac{529}{576}\frac{(\log 2)^2}{d^2}$ for every $d\ge1$.
\end{theorem}

\begin{proof}
We first show that the inverse branch of $F$ fixing $0$ extends
to the entire disk $D(0,m(W))$. Because $W$ is a homeomorphism from $\D$ onto $W(\D)$,
\begin{equation}\label{eq:boundary-W}
\liminf_{|z|\to1}|W(z)|\ge m(W).
\end{equation}
Indeed, if this failed, there would be an $s<m(W)$ and a sequence
$z_j$ with $|z_j|\to1$ such that $|W(z_j)|\le s$.
The closed disk $\overline{D(0,s)}$ is a compact subset of
$W(\D)$, so its inverse image under $W$ is a compact subset of
$\D$, a contradiction.

A finite Blaschke product extends continuously across the unit
circle and has modulus one there. Consequently,
$$
|B(z)|\longrightarrow1\qquad\text{as }|z|\to1.
$$
Together with \eqref{eq:boundary-W}, this implies
\begin{equation}\label{eq:boundary-F}
\liminf_{|z|\to1}|F(z)|\ge m(W).
\end{equation}

Let $0<t<m(W)$, and let $V_t$ be the connected component of $\{z\in\D:|F(z)|<t\}$ containing $0$. By \eqref{eq:boundary-F}, there is an $r_t<1$
such that $|F(z)|>t$ for $r_t<|z|<1$. Hence $V_t$ is relatively compact in $\D$. Moreover,
$|F(z)|=t$ for $z\in\partial V_t$.
For if a boundary point had $|F|<t$, a small connected
neighborhood of that point would lie in $\{|F|<t\}$ and meet
$V_t$, contradicting the definition of $V_t$ as a component.

It follows that $F:V_t\longrightarrow D(0,t)$
is proper, i.e. the inverse image of a compact subset of $D(0,t)$
cannot accumulate on $\partial V_t$. Its image is both open
and closed in $D(0,t)$, so the map is surjective.
Since $F'$ never vanishes, it is a proper local
biholomorphism, hence an unbranched covering.
The target disk is simply connected and $V_t$ is connected,
so this covering has one sheet. Thus $F:V_t\longrightarrow D(0,t)$ is biholomorphic.  Its inverse $f_t:D(0,t)\to V_t\subset\D$ satisfies $f_t(0)=0$, $f_t'(0)=1/B(0)$, and $f_t$ omits every zero $a$ of $B$. 

The Schwarz lemma applied to $\zeta\mapsto f_t(t\zeta)$, a holomorphic self-map of $\D$ fixing $0$, gives $t\,f_t'(0)=t/B(0)\le1$. Letting $t\to m(W)$ yields $B(0)\ge m(W)\ge1/4$ by (T1), hence $\Bee\le4$.

Now let $B$ be nonconstant and $a$ a zero of $B$. We then estimate $1/B(0)$ by using the following omitted-point estimate of Pick \cite{MaMinda}.

\begin{lemma}[Omitted-point estimate]\label{lem:pick}
Suppose that $h:\D\to\D$ is univalent, $h(0)=0$, and $a\in\D\setminus h(\D)$. If $r=|a|>0$, then
$$
|h'(0)|\le\frac{4r}{(1+r)^2}.
$$
\end{lemma}

Applying Lemma \ref{lem:pick} to $\zeta\mapsto f_t(t\zeta)$, we have 
$t/B(0)=tf_t'(0)\le \frac{4|a|}{(1+|a|)^2}$.
 Letting $t\to m(W)$ and using also $m(W)\ge1/4$ from (T1), we have
 $B(0)\ge \frac{(1+|a|)^2}{4|a|}m(W) \ge \frac{(1+|a|)^2}{16|a|}> 1/4$.

Let $V:=1/B(0) <4$ and $r:=|a|$. Then
$$\frac{4r}{(1+r)^2}\ge\frac V4.$$
The equation $\frac{4r}{(1+r)^2}=\frac V4$ has two solutions and it is easy to compute that the smaller positive solution of the equation $\frac{4r}{(1+r)^2}=\frac V4$ is
$$\rho(V)=\frac{2-\sqrt{4-V}}{2+\sqrt{4-V}}=\frac{1-\sqrt{1-V/4}}{1+\sqrt{1-V/4}}   .$$
So, since $r\mapsto 4r/(1+r)^2$ is increasing on $(0,1)$, every zero $a_j$ of $B$ satisfies $|a_j|\ge\rho(V)$. Since $B(0)=\prod_{j=1}^d|a_j|$, we have $1/V\ge\rho(V)^d$ which is equivalent to 
$\frac{\log V}{2d}\le\tanh^{-1}(\sqrt{1-V/4})$. Hence, we have $V\le 4-4\tanh^2\Bigl(\frac{\log V}{2d}\Bigr)$.

If $V\ge2$, then $\log V\ge\log2$ and hence
$V\le 4-4\tanh^2\Bigl(\frac{\log 2}{2d}\Bigr)$.

It is easy to check that $\tanh x>x-\frac{x^3}{3}$ when $x>0$. So for $0<x< (\log 2)/2$, $\tanh x> x-\frac{x^3}{3}=x(1-\frac{x^2}{3})\ge x(1-\frac{(\log2)^2}{12})>\frac{23}{24}x$. Therefore, $V\le 4-4\tanh^2\Bigl(\frac{\log 2}{2d}\Bigr) < 4-\frac{529}{576}\frac{(\log 2)^2}{d^2}$. 
Since $\frac{529}{576}(\log 2)^2<2$ and $d\ge 1$, $\frac{529}{576}\frac{(\log 2)^2}{d^2}<2$ and hence $V< 4-\frac{529}{576}\frac{(\log 2)^2}{d^2}$ when $V<2$ and we are done. It follows that $\Bee_d\le4-\frac{529}{576}\frac{(\log2)^2}{d^2}$. (The strict inequality for $\Bee_d$ asserted in Theorem~\ref{thm:upperN} follows from the attainment result, Corollary~\ref{cor:strict}.)

\end{proof}

\begin{remark}
Theorem~\ref{thm:upper} shows that under the reformulation $\M=\Bee$, Smale's bound $4$ is exactly Koebe's constant: $1/B(0)\le 1/m(W)\le4$. If $1/B(0)$ is close to $4$ then $m(W)$ is close to $1/4$, so by compactness of $\cS$ and the uniqueness part of Koebe's theorem $W$ is close (locally uniformly) to a rotation of the Koebe function, and all zeros of $B$ are close to $\T$ with $\sum_j(1-|a_j|)\approx\log 4$; in particular $\deg B\to\infty$. Since $f_t$ omits all zeros of $B$ simultaneously, one may expect that the hyperbolic metric of $\D\setminus\{a_1,\dots,a_d\}$ gives a better bound than $4-C/d^2$. See Section~\ref{sec:cr}.
\end{remark}

\subsection{Degree one}\label{sec:grunsky}

\begin{theorem}[Degree one]\label{thm:sqrt75}
Let $W\in\cS$ and let $B(z)=\lambda(z-a)/(1-\overline a\,z)$ be a degree-one
Blaschke factor. If $(W,B)$ is an admissible pair, then
$$
|a|>\sqrt{\frac57}.
$$
Consequently
$$
\frac1{|B(0)|}<\sqrt{\frac75}
\qquad\text{and}\qquad
\Bee_1\le\sqrt{\frac75}.
$$
(The strict inequality $\Bee_1<\sqrt{7/5}$ follows from Corollary~\ref{cor:strict}.)
\end{theorem}

\begin{proof}
Since $B(0)>0$ and $a$ is a zero of $B$, $a\neq 0$. So we can write $a=se^{i\varphi}$ with $0<s<1$.
The rotated function $\widetilde W(z)=e^{-i\varphi}W(e^{i\varphi}z)$ still lies in $\cS$, and
$$
B(e^{i\varphi}z)=-\lambda e^{i\varphi}B_s(z),\qquad
B_s(z)=\frac{s-z}{1-sz}.
$$
Thus $(WB)'$ is zero-free if and only if $(\widetilde W B_s)'$ is zero-free.
It is therefore enough to treat $B_s$, and enough to show that no pair $(W,B_s)$ with $s\le\sqrt{5/7}$ is admissible.

Recall that $g(z)=zW'(z)/W(z)$, so that $g(0)=1$.
Univalence implies that $g$ is holomorphic and zero-free on $\D$, so there is a unique holomorphic logarithm $\ell$ of $g$ with $\ell(0)=0$.
Grunsky's inequality~\cite{Grunsky1932} (see also \cite[Chapter~3]{Duren}) gives
\begin{equation}\label{eq:grunsky}
\bigl|\ell(z)\bigr|
\le
L(|z|),\qquad
L(r)=\log\frac{1+r}{1-r}.
\end{equation}
Notice that $L$ is increasing on $(0,1)$.

The logarithmic derivative of the Blaschke factor gives
$$
\beta_s(z)
:=\frac{zB_s'(z)}{B_s(z)}
=\frac{(1-s^2)z}{(z-s)(1-sz)}.
$$
Admissibility of $(W,B_s)$ implies that $\frac{zF'}{F}=g+\beta_s$ has no zero in $\D$ where $F=WB_s$.
To test this, we cancel the poles of $\beta_s$ and consider $(z-s)(1-sz)\,g(z)+(1-s^2)z$ and then its holomorphic homotopy
\begin{equation}\label{eq:Pt}
P_t(z)=(z-s)(1-sz)\,e^{t\ell(z)}+(1-s^2)z,\qquad 0\le t\le 1.
\end{equation}
Each $P_t$ is holomorphic in $\D$.
A zero of $P_1$ is a zero of $(WB_s)'$, because
$$
(WB_s)'(z)=-\frac{W(z)}{z}\,\frac{P_1(z)}{(1-sz)^2},
$$
the identity at the origin being understood by continuity, and because $W(z)/z$ does not vanish in $\D$.

For $w\in\C\setminus\{0\}$ write $d(w)$ for the least modulus of a logarithm of $w$:
$$
d(w)=\sqrt{(\log|w|)^2+(\Arg w)^2},\qquad -\pi<\Arg w\le\pi.
$$
(The two choices $\pm\pi$ on the negative axis give the same value, so $d$ is continuous on $\C\setminus\{0\}$.)
The comparison on $|z|=s$ needs only this least modulus, not a holomorphic branch of $\Log$ along the whole circle.

Let $\gamma$ be the circle $|z|=s$, traversed once positively.
The point $z=s$ is not a zero of any $P_t$, since $P_t(s)=s(1-s^2)\ne0$.
At every other point of $\gamma$ one has $P_t(z)=0$ if and only if $e^{t\ell(z)}=-\beta_s(z)$.
Combined with~\eqref{eq:grunsky} and $0\le t\le1$, this implies that if
\begin{equation}\label{eq:crit}
d\bigl(-\beta_s(z)\bigr)>L(s)
\qquad\bigl(|z|=s,\ z\ne s\bigr),
\end{equation}
then $P_t$ cannot vanish on $\gamma$.

As long as $P_t$ has no zero on $\gamma$, the winding number of $P_t\circ\gamma$ about $0$ is well-defined.
That winding number equals the number of zeros of $P_t$ inside $\gamma$, counted with multiplicity.
It depends continuously on $t$, because $P_t(z)$ varies continuously with $t$, uniformly for $z\in\gamma$, and never hits $0$ on $\gamma$.
An integer-valued continuous function of $t\in[0,1]$ is constant.
Thus the number of zeros of $P_t$ inside $\gamma$ is independent of $t$.

At $t=0$,
$$
P_0(z)=-sz^2+2z-s.
$$
The roots are $\bigl(1\pm\sqrt{1-s^2}\bigr)/s$.
The larger root is at least $1/s>1$, so it lies outside $\gamma$.
For the smaller root, $\sqrt{1-s^2}<1$ rearranges to $1-\sqrt{1-s^2}<s^2$, hence the root lies in $(0,s)$.
So $P_0$ has exactly one zero inside $\gamma$, and therefore so does $P_1$.
So whenever ~\eqref{eq:crit} holds, the pair $(W,B_s)$ cannot be admissible. We will show that ~\eqref{eq:crit} holds for all $0<s\le s_0:=\sqrt{\frac57}$. The number $s_0$ is chosen so that $\cosh L(s)$ is an integer at $s=s_0$. In fact, $\cosh L(s_0)=6$ and $L(s_0)<\pi$.

To begin with, let $s\in(0,1)$ and we try to get lower estimates of $d\bigl(-\beta_s(z)\bigr)$ by 
dividing the circle $\gamma=\{|z|=s\}$ into the upper open semicircle, the lower open semicircle, and the point $z=-s$. The point $z=s$ has already been excluded, since $P_t(s)\ne0$.
On the upper semicircle $z=se^{i\theta}$ with $0<\theta<\pi$, a direct computation gives
$$
\frac1{\beta_s(z)}=c(1-\cos\theta)+i\sin\theta,
\qquad
c:=\cosh L(s)=\frac{1+s^2}{1-s^2}.
$$
The right-hand side lies in the open first quadrant, so $-\beta_s(z)$ lies in the open second quadrant.
The substitution $x=c\tan(\theta/2)$ runs through $(0,\infty)$ as $\theta$ runs through $(0,\pi)$, and gives
\begin{equation}\label{eq:Jb}
d\bigl(-\beta_s(se^{i\theta})\bigr)
=\sqrt{A\bigl(L(s),x\bigr)^2+b(x)^2},
\end{equation}
where
$$
A(L,x)=\log\frac{\cosh^2L+x^2}{2x\cosh L\sqrt{1+x^2}},
\qquad
b(x)=\frac\pi2+\arctan x
$$
and we also let $\mathcal J\bigl(L(s),x\bigr)
:=\sqrt{A\bigl(L(s),x\bigr)^2+b(x)^2}$ (not to be confused with $J(x,y)$ of Section~\ref{sec:quartic}).

The lower semicircle of $\gamma$ repeats these values by conjugation. If $z$ lies on the upper semicircle, then $\beta_s(\overline z)=\overline{\beta_s(z)}$, so $-\beta_s(\overline z)$ is the conjugate of $-\beta_s(z)$. The least-modulus $d$ depends only on the modulus and on the absolute value of the argument, hence
$$d\bigl(-\beta_s(\overline z)\bigr)=d\bigl(-\beta_s(z)\bigr).$$
No new values occur on the lower semicircle.
The remaining point is $z=-s$. There
$$d\bigl(-\beta_s(-s)\bigr)=\sqrt{\log^2(2c)+\pi^2}\ge\pi.$$ Since $L(s)$ is increasing on $(0,1)$, for $s\le s_0$ we have $L(s)\le L(s_0)=2.47788\ldots <\pi$.
Hence, ~\eqref{eq:crit} holds at $z=-s$ for all $0<s\le s_0$.

At $s=s_0$ write $L_0=L(s_0)$, $A(x)=A(L_0,x)$ and $E(x)=A(x)^2+b(x)^2$, so that $\mathcal J(L_0,x)=\sqrt{E(x)}$.
To get a lower estimate of $E$, we will try to get a lower estimate of $A+kb$ by using the Cauchy--Schwarz inequality $A^2+b^2\ge\frac{(A+kb)^2}{1+k^2}$. The trick is to choose $k$ so that $A+kb$ is monotone and hence one can estimate its minimum at the end points of some intervals.

Differentiating $A$ and $b$ gives
$$
-(1+x^2)A'(x)
=\frac1x+\frac{70x}{36+x^2},
\qquad
(1+x^2)b'(x)=1.
$$

On the interval $0<x\le1$, we have $36+x^2\le37$, so
$$
-(1+x^2)A'(x)
\ge\frac1x+\frac{70}{37}x
\ge2\sqrt{\frac{70}{37}}.
$$
The last step follows from the AM--GM inequality. Hence,
$$-(1+x^2)A'(x)\ge2\sqrt{\frac{70}{37}}>\frac{11}{4}=(1+x^2)b'(x)\cdot\frac{11}{4}$$
so that by taking $k=11/4$, we have $(A+(11/4)b)'<0$ for $0<x\le 1$ which means $A+(11/4)b$ has a minimum at $x=1$. Since $A_1:=A(1)=\log\frac{37}{12\sqrt2}$ and $b(1)=3\pi/4$, we have for $0<x\le 1$, $$E(x)\ge\frac{\bigl(A(1)+(11/4)b(1)\bigr)^2}{1+(11/4)^2}=\frac{\bigl(4\log\frac{37}{12\sqrt2}+33\pi/4\bigr)^2}{137}=6.1538\ldots.$$

On the next interval one wants $A+kb$ increasing, so that the same endpoint $x=1$ is again the minimum.
The slope $k=10/3$ and the cut $x=13/10$ are chosen together to make this elementary.
One finds
$$
x(1+x^2)\Bigl(A+\frac{10}3 b\Bigr)'
=\frac{2x^2(1+x^2)}{36+x^2}-1-2x^2+\frac{10}3 x.
$$
The first summand is at least $4/37$ for $x\ge1$, with equality at $x=1$.
The quadratic $-1-2x^2+(10/3)x$ has its maximum at $x=5/6<1$, so it is decreasing on $[1,13/10]$, and its value at the right endpoint is $-7/150$.
Hence the derivative is at least
$$
\frac4{37}-\frac7{150}=\frac{341}{5550}>0.
$$
Thus $A+(10/3)b$ is increasing on $[1,13/10]$, its minimum is $A_1+(10/3)(3\pi/4)=A_1+5\pi/2$, and
\begin{equation}\label{eq:I2}
E(x)\ge\frac{\bigl(3A_1+15\pi/2\bigr)^2}{109}=6.1543\ldots
\qquad(1\le x\le13/10).
\end{equation}
The cut $13/10$ is also large enough that, past it, $b^2(x)$ alone already exceeds $L_0^2=6.1399\ldots$. In fact, since $b$ is increasing,
\begin{equation}\label{eq:I3}
E(x)\ge b(x)^2\ge b\Bigl(\frac{13}{10}\Bigr)^2
=(2.48589\ldots)^2=6.1796\ldots
\qquad\bigl(x\ge13/10\bigr).
\end{equation}

Combining all the three lower bounds of $E$, we have $E(x)>L_0^2$ for $x>0$ so that
$\mathcal J(L_0,x)>L_0$ and~\eqref{eq:crit} holds at $s=s_0$. It remains to transfer this inequality to every smaller $s$ by showing that ${\mathcal J}(L,x)-L$ does not increase with $L$.

Differentiating $A$ in the first variable gives
$$
\frac{\partial A}{\partial L}(L,x)
=\tanh L\cdot\frac{\cosh^2L-x^2}{\cosh^2L+x^2},
$$
hence $\bigl|\partial A/\partial L\bigr|\le1$ and $|A(L,x)-A(M,x)|\le|L-M|$.
The quantity $\mathcal J(L,x)$ is the Euclidean norm of $\bigl(A(L,x),b(x)\bigr)$, and $b$ is independent of $L$, so
\begin{equation}\label{eq:Lip}
\bigl|\mathcal J(L,x)-{\mathcal J}(M,x)\bigr|\le|L-M|.
\end{equation}
From (\ref{eq:Lip}), we have
$${\mathcal J}(L,x)\ge {\mathcal J}(L_0,x)-|L-L_0|.$$
If $L\le L_0$, then $|L-L_0|=L_0-L$, so
$${\mathcal J}(L,x)\ge {\mathcal J}(L_0,x)-(L_0-L)=\mathcal J(L_0,x)-L_0+L.$$
Subtracting $L$ from both sides,
$${\mathcal J}(L,x)-L\ge \mathcal J(L_0,x)-L_0>0.$$
 Therefore ${\mathcal J}(L,x)>L$ for every $x>0$ and every $L\le L_0$. Since $L$ is increasing on $(0,1)$, for $s\le s_0$ one has $L(s)\le L_0$ and hence ${\mathcal J}(L(s),x)>L(s)$ which is the required criterion ~\eqref{eq:crit} for non-admissibility.
We then conclude that every admissible degree-one pair satisfies $s>s_0=\sqrt{5/7}$ so that $1/B(0)=1/B_s(0)=1/s {<1/s_0}$ and hence $\Bee_1\le1/s_0=1.1832\ldots$. 
\end{proof}

\section{Every degree occurs}\label{sec:every}
In this section, we show that the extremal problems defining the constants $\Bee_{d}$ ($d \geq 2$) are not vacuous. That is, for each $d \geq 2$, there exists an admissible Blaschke pair $(W,B)$ with $\deg(B) = d$. In fact, we show that for each $d \geq 2$, there exists a \textbf{regular} admissible Blaschke pair $(W,B)$ with $\deg(B) = d$. This will be established in Theorem \ref{thm:nonempty} as a direct consequence of Lemmas \ref{lem:wdrop} and \ref{lem:domain} below. We also characterise, up to a degenerate case, the functions $W$ that carry a nonconstant factor $B$ (Theorem~\ref{thm:char}), prove that each $\Bee_d$ is attained (Theorem~\ref{thm:attain}), and deduce the strict inequalities of Theorem~\ref{thm:upperN} (Corollary~\ref{cor:strict}).

We first fix some notation to be used in the statement of Lemma \ref{lem:wdrop}. Let $W$ be univalent on a neighbourhood of $\cl\D$, with $W(0)=0$ and $W'\ne0$ on $\cl\D$. On $\T$ set
$$
u(t)=\log\bigl|W(e^{it})\bigr|,
\qquad
v(t)=\frac{d}{dt}\arg W(e^{it}),
$$
and express the logarithmic derivative of $W$ along the unit circle as
\begin{equation}\label{eq:h}
h(t)=u'(t)+iv(t)=ie^{it}\frac{W'(e^{it})}{W(e^{it})}=i\,g(e^{it}).
\end{equation}
Then $h$ is real analytic and zero-free. The curve $t \rightarrow ie^{it} : [0,2 \pi] \rightarrow \mathbb{C}$ has winding number $1$ with respect to the origin, while $W(\T)$ has winding number $1$ with respect to the origin, and $W'(\T)$ has winding number $0$ with respect to the origin, so that $\wind_0(h)=1+0+(-1)=0$. Here and in what follows, by $$\wind_0(\gamma)$$ we mean the winding number of the curve parameterized by $\gamma : [0, 2 \pi] \rightarrow \mathbb{C} \setminus \{0 \}$ with respect to the origin.

On the unit circle $|z|=1$ a single Blaschke factor
$$
b_a(z)=\frac{|a|}{a}\frac{a-z}{1-\overline{a}z},\qquad a=re^{i\theta},
$$
satisfies $|b_a|=1$. A direct computation gives the below quantity is real and positive on $\partial\mathbb{D}$,
$$
\frac{z b_a'(z)}{b_a(z)}=\frac{1-|a|^2}{|z-a|^2}\qquad\bigl(|z|=1\bigr).
$$
Parametrizing $z=e^{it}$ yields the real angular derivative
$$
\frac{d}{dt}\arg b_a(e^{it})=\frac{1-r^2}{|e^{it}-re^{i\theta}|^2}.
$$
(The unimodular constant that normalizes $b_a(0)>0$ does not affect the derivative.)  

For the full Blaschke product $B_r=\prod_{j=1}^n b_{a_j}$ with distinct $a_j=re^{it_j}$, the angular derivatives add and we have
$$
S_r(t)=\frac{d}{dt}\arg B_r(e^{it})=\sum_{j=1}^n\frac{1-r^2}{|e^{it}-re^{it_j}|^2}
$$
and $S_r(t) \to 0$ as $r \to 1$ when $t\neq t_i$. 

In order to build an admissible pair $(W,B_r)$, we want the derivative of $F_r=WB_r$ has no zero in $\mathbb{D}$ . On the unit circle one has the identity
$$ie^{it}\frac{F_r'(e^{it})}{F_r(e^{it})}=h(t)+iS_r(t).$$

The function $F_r$ is holomorphic on a neighbourhood of
$\overline{\mathbb D}$, has no zeros on $\mathbb T$, and has exactly
$n+1$ zeros in $\mathbb D$, counted with multiplicity. Therefore, if
$F_r'$ is nonvanishing on $\mathbb T$, the argument principle yields
$$
\operatorname{wind}_0\left(
t\longmapsto ie^{it}\frac{F_r'(e^{it})}{F_r(e^{it})}
\right)
=1+N_{\mathbb D}(F_r')-(n+1)
=N_{\mathbb D}(F_r')-n,
$$
where $N_{\mathbb D}(F_r')$ is the number of zeros of $F_r'$ in $\D$.
It thus suffices to prove, for $r$ sufficiently close to $1$, that
$h+iS_r$ is nonvanishing and has winding number $-n$ so that $N_{\mathbb D}(F_r')=0$.

The following lemma suggests one situation that we can find  
$\wind_0(h+iS_r)=-n$
and hence $N_{\mathbb D}(F_r')=0$.

\begin{lemma}[Winding drop]\label{lem:wdrop}
Suppose $t_1,\dots,t_n$ are distinct points at which the conditions \begin{equation}\label{eq:min}
u'(t_j)=0,\qquad u''(t_j)>0,\qquad v(t_j)<0.
\end{equation}
are satisfied (note that these conditions may be satisfied at additional points on $\mathbb{T}$, i.e. there is \textbf{no requirement} that $\{t_1, \ldots , t_n \}$ is the entire set of points at which (\ref{eq:min}) holds). Let $B_r$ be the degree-$n$ Blaschke product with zeros $re^{it_j}$ and $B_r(0)=r^n>0$, and write
$$
S_r(t)=\sum_{j=1}^n\frac{1-r^2}{|e^{it}-re^{it_j}|^2}
=\frac{d}{dt}\arg B_r(e^{it}).
$$
For all $r<1$ sufficiently close to $1$, the product $F_r=WB_r$ satisfies $F_r'\ne0$ on $\cl\D$. In particular $(W/W'(0),B_r)$ is a regular admissible pair of degree $n$.
\end{lemma}

\begin{proof} 

 Replacing the parameter interval $[0,2\pi]$ by $[\tau,\tau+2\pi]$ for a suitable $\tau$, we may assume that no $t_j$ is an endpoint of the interval; we keep writing $[0,2\pi]$.
 Denote the stationary point set of $u$ by $E \subset [0, 2 \pi]$.
 $E$ is a finite set since $u$ is real-analytic and $u' \not\equiv 0$ (as $u''(t_j)>0$). If at some point $z = e^{it} \in \mathbb{T}$ we had $F_r'(e^{it}) = 0$, then $h(t) + iS_r(t) = u'(t) + i(v(t) + S_r(t)) = 0$ so that $v(t) + S_r(t)=0$ and $$t \in E.$$ 
 Suppose $t \in E \setminus \{t_1, \ldots , t_n \}$, then because $h$ is zero free, we have $v(t)\neq 0$ and hence for $r$ sufficiently close to $1$, $S_r(t)< |v(t)|$ and therefore $v(t) + S_r(t) \neq 0$ which is a contradiction. It follows that we must have $t = t_j$ for some $j=1, \ldots , n$, but in this case we have that $$v(t_j) + S_r(t_j) > 0$$ for all $r$ sufficiently close to $1$. We therefore conclude that $F_r'$ is non-vanishing on $\mathbb{T}$ for all $r$ sufficiently close to $1$. It remains to show that for $r$ sufficiently close to $1$, we have $$\wind_0(i e^{it} \frac{F_r'(e^{it})}{F_r(e^{it})}) = -n.$$ 
 
 To this end, consider the linear homotopy between the loops $h$ and $h + iS_r$, given by $$H_r(t, \lambda) = h(t) + i \lambda S_r(t) : [0,2 \pi] \times [0, 1] \rightarrow \mathbb{C}$$ 
 
 Considering the real and imaginary parts of $H$, analogously to the argument produced to prove that $F_r'(z) \neq 0 ; z \in \mathbb{T}$, 
 we see that if $H_r(t,\lambda)=0$, then $t\in E$.
 For any
$t\in E\setminus\{t_1,\ldots,t_n\}$, like before, we can choose  $r$ so close to $1$ that
$$
|\lambda S_r(t)|\le S_r(t)<|v(t)|,
$$
so that $H_r(t,\lambda)\ne0$ for every $0\le\lambda\le1$.

On the other hand, we can choose $r$ so close to $1$ that each of the values $$-\frac{v(t_j)}{S_r(t_j)} \qquad  (j=1,...,n)$$ belong to $(0,1)$ because
 $$
S_r(t_j)\ge\frac{1+r}{1-r}\longrightarrow+\infty.
$$
 Then at each parameter $t_j$, there is exactly one zero of $H_r(t_j,\lambda)$,
occurring at
$$
\lambda_j=-\frac{v(t_j)}{S_r(t_j)}\in(0,1).
$$
Thus the zeros of $H_r$ on $[0,2\pi]\times[0,1]$ are precisely
$$
(t_j,\lambda_j),\qquad j=1,\ldots,n.
$$
In particular, both boundary loops $H_r(\cdot,0)$ and
$H_r(\cdot,1)$ are nonvanishing. Since $F_r$ is nonvanishing on
$\mathbb T$, this also proves that $F_r'$ is nonvanishing on
$\mathbb T$.

We next compute the winding change. Regard $H_r$ as a map
into $\mathbb R^2$ through its real and imaginary parts.
At each zero,
$$
\det D_{(t,\lambda)}H_r(t_j,\lambda_j)
=
\det
\begin{pmatrix}
u''(t_j)&0\\
v'(t_j)+\lambda_jS_r'(t_j)&S_r(t_j)
\end{pmatrix}
=u''(t_j)S_r(t_j)>0.
$$
Thus every zero is nondegenerate and has local degree $+1$.

Give $R:=[0,2\pi]\times[0,1]$ its usual orientation, with $t$ pointing right
and $\lambda$ pointing up. Its outer boundary is therefore
traversed counterclockwise. Remove disjoint small disks around
the $n$ zeros $\alpha_j:=(t_j,\lambda_j)$, contained in the interior of $R$. At each zero $\alpha_j$, we have $\det DH_r(\alpha_j)>0$.
By the inverse function theorem, $H_r$ is an
orientation-preserving diffeomorphism on a sufficiently small
neighborhood of $\alpha_j$. Choose the removed disk $D_j$ with
its closure contained in this neighborhood.
Then $H_r$ maps $\partial D_j$ to a simple closed curve
enclosing $H_r(\alpha_j)=0$. When $\partial D_j$ is traversed
clockwise, its image is also traversed clockwise, exactly once,
and hence has winding number $-1$.

The outer boundary of the rectangle is traversed counterclockwise.
Along its lower edge, $\lambda=0$ and $t$ increases from
$0$ to $2\pi$. Since
$$
H_r(t,0)=h(t),
$$
the image of this edge traces the loop $h$ in its defining
direction, with change of argument
$$
2\pi\,\operatorname{wind}_0(h).
$$
Along the upper edge, $\lambda=1$ and $t$ decreases from
$2\pi$ to $0$. Since
$$
H_r(t,1)=h(t)+iS_r(t),
$$
the image traces the loop $h+iS_r$ in the reverse direction,
with change of argument
$$
-2\pi\,\operatorname{wind}_0(h+iS_r).
$$
Finally, periodicity gives
$$
H_r(0,\lambda)=H_r(2\pi,\lambda),
\qquad 0\le\lambda\le1.
$$
On the right edge, $\lambda$ increases from $0$ to $1$,
whereas on the left edge it decreases from $1$ to $0$.
Thus their images trace the same path in opposite directions,
and their changes of argument cancel.

Since $d\arg H_r$ is a closed $1$-form on $R$ minus the zeros of $H_r$, by Stokes' theorem, the total change of argument along the
oriented boundary of the region with the $n$ disks removed is zero.
Adding the contributions above and dividing by $2\pi$, we obtain
$$
\operatorname{wind}_0(h)
-\operatorname{wind}_0(h+iS_r)-n=0.
$$
Since $\operatorname{wind}_0(h)=0$, it follows that $\operatorname{wind}_0(h+iS_r)=-n$. 

Applying the previously established argument-principle identity
for the holomorphic function $F_r$, we now obtain
$$
N_{\mathbb D}(F_r')-n=-n,
$$
and hence $N_{\mathbb D}(F_r')=0$.
Together with the nonvanishing of $F_r'$ on $\mathbb T$,
this proves
$$
F_r'(z)\ne0
\qquad(z\in\overline{\mathbb D})
$$
for every $r$ sufficiently close to $1$.

\end{proof}

\begin{remark}[The condition \eqref{eq:min} in terms of $g$]\label{rem:g}
By \eqref{eq:h}, $u'(t)=-\Ima g(e^{it})$, $v(t)=\Rea g(e^{it})$ and $u''(t)=-\frac{d}{dt}\Ima g(e^{it})$. Hence \eqref{eq:min} at $t_j$ says exactly that
$$
g(e^{it_j})\in(-\infty,0)\qquad\text{and}\qquad \frac{d}{dt}\Ima g(e^{it})\Big|_{t=t_j}<0 ,
$$
i.e.\ the curve $t\mapsto g(e^{it})$ crosses the negative real axis at $t_j$ transversally from the upper to the lower half-plane (a \emph{downward crossing}). Lemma~\ref{lem:wdrop} thus reads: \emph{if $g(\T)$ has $n$ downward crossings of $(-\infty,0)$, then $W$ carries a regular admissible pair of degree $n$.} In particular a starlike $W$ ($\Rea g>0$) carries none. Theorem~\ref{thm:char} below gives the converse.
\end{remark}

\begin{remark}\label{rem:F}
The proof of Lemma~\ref{lem:wdrop} uses only that $F=W$ is holomorphic on a neighbourhood of $\cl\D$, zero-free on $\T$, with $F'\neq0$ on $\cl\D$, together with the value of $\wind_0(h)$, $h(t)=ie^{it}F'(e^{it})/F(e^{it})$. It therefore applies verbatim to $F=WB$ for a regular admissible pair; see Proposition~\ref{prop:increase}.
\end{remark}

\begin{lemma}[Lobed domain]\label{lem:domain}
For every integer $n\ge1$ there exists $W\in\cS$, univalent
on a neighbourhood of $\cl\D$, whose functions $u$ and $v$
satisfy \eqref{eq:min} at (at least) $n$ distinct parameters.
\end{lemma}

\begin{proof}
We first construct a piecewise-analytic Jordan domain $\Omega$
containing the origin, with $n$ distinguished points on circular
boundary arcs where the modulus has a nondegenerate local
minimum and the argument is decreasing in the positive boundary
direction.

Choose $n$ equally spaced directions $\theta_j=2\pi j/n$.
Take the central disk $D(0,1)$ and the peripheral disks
$D(Le^{i\theta_j},1)$, where $L$ is sufficiently large.
Join the central disk to each peripheral disk by a thin straight
corridor around the segment from $0$ to
$e^{i\theta_j}(L+i/2)$. Both ends of each corridor lie inside
the corresponding disks. For sufficiently narrow corridors,
the different attachments occupy disjoint angular sectors
outside the central disk. Their union is therefore a Jordan
domain with finitely many corners.

The points
$$
p_j=(L-1)e^{i\theta_j}
$$
remain on exposed circular boundary arcs: their distance from
the corresponding corridor's centre line tends to $1/2$ as
$L\to\infty$, so sufficiently narrow corridors avoid them.
Figure~\ref{fig:lobe} illustrates this construction for $n=2$.

%%%%%%%%%%

\begin{figure}[ht]
\centering
\begin{tikzpicture}[
  scale=0.65,
  >=stealth,
  every path/.style={},
  every node/.style={font=\small}
]

% Fill the connecting strip.
\fill[purple!20]
  (-8,-0.6) -- (8,0.4) --
  (8,0.6) -- (-8,-0.4) -- cycle;

% Fill the three disks.
\fill[orange!30] (-8,0) circle (1);
\fill[orange!30] (8,0) circle (1);
\fill[blue!15] (0,0) circle (1);

% Draw circular boundaries outside the strip.
% Use PGF clipping directly, without a TikZ \clip command.
\begin{scope}
  \pgfseteorule

  % Outer rectangle.
  \pgfpathrectangle
    {\pgfpoint{-10cm}{-2cm}}
    {\pgfpoint{20cm}{4cm}}

  % Strip to exclude.
  \pgfpathmoveto{\pgfpoint{-8cm}{-0.6cm}}
  \pgfpathlineto{\pgfpoint{8cm}{0.4cm}}
  \pgfpathlineto{\pgfpoint{8cm}{0.6cm}}
  \pgfpathlineto{\pgfpoint{-8cm}{-0.4cm}}
  \pgfpathclose

  \pgfusepath{clip}

  \draw[thick, orange!70!black]
    (-8,0) circle (1);
  \draw[thick, orange!70!black]
    (8,0) circle (1);
  \draw[thick, blue!60!black]
    (0,0) circle (1);
\end{scope}

% Draw strip boundaries outside the disks.
\begin{scope}
  \pgfseteorule

  % Outer rectangle.
  \pgfpathrectangle
    {\pgfpoint{-10cm}{-2cm}}
    {\pgfpoint{20cm}{4cm}}

  % Three disks to exclude.
  \pgfpathcircle{\pgfpoint{-8cm}{0cm}}{1cm}
  \pgfpathcircle{\pgfpoint{0cm}{0cm}}{1cm}
  \pgfpathcircle{\pgfpoint{8cm}{0cm}}{1cm}

  \pgfusepath{clip}

  \draw[thick, purple!65!black]
    (-8,-0.6) -- (8,0.4) --
    (8,0.6) -- (-8,-0.4) -- cycle;
\end{scope}

% Origin and central disk.
\fill[black] (0,0) circle (0.045);
\node[below] at (0,-0.05) {$0$};
\node at (0,-1.4) {unit disk};

% Distinguished points.
\fill[red!75!black] (-7,0) circle (0.09);
\fill[red!75!black] (7,0) circle (0.09);

\node[red!75!black, above right]
  at (-7,-0.2) {$p_1$};
\node[red!75!black, above left]
  at (7,-0.5) {$p_2$};

% Positive boundary orientation.
\draw[->, thick]
  (-8,0) ++(8:1)
  arc[start angle=8, end angle=38, radius=1];

\draw[->, thick]
  (8,0) ++(188:1)
  arc[start angle=188, end angle=218, radius=1];

\end{tikzpicture}
\caption{A model lobed domain $\Omega$ for $n=2$.
All three disks have radius $1$.}
\label{fig:lobe}
\end{figure}
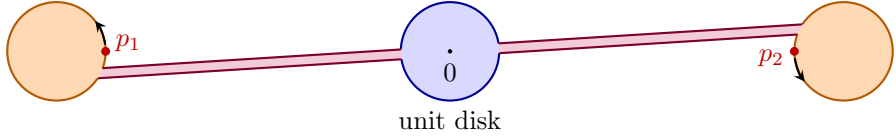

Near $p_j$, the positively oriented circular arc is parametrized
by $\gamma_j(s)=e^{i\theta_j}(L-e^{is})$. At $s=0$,
$$
\frac{d}{ds}\log|\gamma_j(s)|=0,\qquad
\frac{d^2}{ds^2}\log|\gamma_j(s)|
=\frac{L}{(L-1)^2}>0,\qquad
\frac{d}{ds}\arg\gamma_j(s)=-\frac1{L-1}<0.
$$
Thus these arcs have the required properties.

Let $f:\D\to\Omega$ be the Riemann map normalized by
$f(0)=0$ and $f'(0)>0$, and write $f(e^{it_j})=p_j$.
By boundary correspondence and Schwarz reflection, $f$
extends conformally across each distinguished circular arc.
Its boundary parametrization preserves orientation  and has
nonzero speed, so the preceding strict inequalities give
$$
u_f'(t_j)=0,\qquad u_f''(t_j)>0,\qquad v_f(t_j)<0,
$$
where $u_f(t)=\log|f(e^{it})|$ and
$v_f(t)=\operatorname{Im}(ie^{it}f'(e^{it})/f(e^{it}))$.

Since $\Omega$ has corners, $f$ need not extend univalently across the whole of $\T$ and we therefore pass to a dilation. Near each $(1,t_j)$, define
$$
G(\rho,t)
=\frac{\partial \log|f(\rho e^{it})|}{\partial t}=\operatorname{Re}
 \left(i\rho e^{it}
       \frac{f'(\rho e^{it})}{f(\rho e^{it})}\right).
$$
This is real analytic there, and
$$
G(1,t_j)=0,
\qquad
\frac{\partial G}{\partial t}(1,t_j)
=u_f''(t_j)>0.
$$
The implicit function theorem therefore gives a nearby
parameter $t_j(\rho)$ satisfying
$$
G(\rho,t_j(\rho))=0
$$
for $\rho<1$ sufficiently close to $1$. By continuity, the positive second derivative and the negative
angular derivative persist at these points. Since there are only finitely many $t_j$,
we may choose one $\rho\in(0,1)$ for which all these
properties hold and the parameters $t_j(\rho)$ remain
distinct.

Set
$$
W(z)=\frac{f(\rho z)}{\rho f'(0)}.
$$
Then $W(0)=0$, $W'(0)=1$, and $W$ is univalent on
$D(0,1/\rho)$, a neighbourhood of $\cl\D$.
Thus $W\in\cS$.
Multiplication by the positive constant
$1/(\rho f'(0))$ does not change the relevant angular
derivatives, so $W$ satisfies \eqref{eq:min} at all
$t_j(\rho)$.
\end{proof}

\begin{theorem}\label{thm:nonempty}
For every integer $n\ge0$ there is a regular admissible pair of degree exactly $n$.
\end{theorem}

\begin{proof}
The pair $(z,1)$ covers $n=0$. For $n\ge1$, this is an immediate consequence of Lemmas \ref{lem:wdrop} and \ref{lem:domain}.
\end{proof}

\begin{remark}[Symmetric pairs and the $n$-th root transform]\label{rem:sym}
Lemma~\ref{lem:domain} needs only to be proved for $n=1$. Indeed, for $f\in\cS$ let $f_n(z):=f(z^n)^{1/n}$ (the branch with $f_n'(0)=1$); then $f_n\in\cS$, $f_n$ is univalent on a neighbourhood of $\cl\D$ if $f$ is, and $g_{f_n}(z)=g_f(z^n)$. Hence a downward crossing of $g_f(\T)$ at $e^{it_0}$ produces $n$ downward crossings of $g_{f_n}(\T)$ at the $n$-th roots of $e^{it_0}$, and Lemma~\ref{lem:wdrop} yields a regular admissible pair $(f_n,B)$ of degree $n$ with $B(z)=\beta(z^n)$, $\beta$ a single Blaschke factor. More generally, for a finite Blaschke product $\beta$ with $\beta(0)>0$ and $B(z):=\beta(z^n)$ one has $zB'(z)/B(z)=n\,(w\beta'(w)/\beta(w))|_{w=z^n}$, so
$$
(f_n,\beta(z^n))\ \text{is admissible}\iff \beta\ \text{has simple zeros and}\ g_f+n\,\frac{w\beta'}{\beta}\ \text{has no zero in }\D ,
$$
and $B(0)=\beta(0)$. Thus $\Bee_{nk}\ge\Bee^{(n)}_k:=\sup\{1/\beta(0)\}$ over such $(f,\beta)$ with $\deg\beta\le k$, and $\Bee^{(1)}_1=\Bee_1$. The method of Theorem~\ref{thm:sqrt75} applies to $\Bee^{(n)}_1$ with $\beta_s$ replaced by $n\beta_s$. Under Proposition~\ref{prop:poly-to-pair}, polynomials of the form $P(w)=w\,p(w^n)$ (such as Smale's example $w-w^d/d$) correspond exactly to such symmetric pairs.
\end{remark}

\begin{theorem}\label{thm:char}
Let $W$ be univalent on a neighbourhood of $\cl\D$ with $W(0)=0$, and let $g=zW'/W$.
\begin{enumerate}
\item[(a)] If $(W,B)$ is admissible and $B$ is nonconstant with zeros $a_1,\dots,a_d$, then there is $\zeta\in\T$ with
$$
g(\zeta)\in(-\infty,0)\qquad\text{and}\qquad -g(\zeta)\ \ge\ \sum_{j=1}^d\frac{1-|a_j|}{1+|a_j|}.
$$
In particular $g(\T)\cap(-\infty,0)\neq\emptyset$, and with
$\nu(W):=\max\{-g(\zeta):\ \zeta\in\T,\ g(\zeta)\in(-\infty,0)\}$ one has
$$
\log\frac1{B(0)}\ \le\ 5\,\nu(W).
$$
\item[(b)] Conversely, if $t\mapsto g(e^{it})$ has $n\ge1$ downward crossings of $(-\infty,0)$ (Remark~\ref{rem:g}), then $W/W'(0)$ carries a regular admissible pair of degree $n$. In particular, if $g(\T)\cap(-\infty,0)\ne\emptyset$ and every point of $g(\T)\cap(-\infty,0)$ is a transversal crossing, then $W/W'(0)$ carries a nonconstant Blaschke factor.
\end{enumerate}
\end{theorem}

\begin{proof}
(a) Let $F=WB$ and $\Phi=zF'/F=g+zB'/B$. Since $W$ is univalent on a neighbourhood of $\cl\D$, $g$ is holomorphic and zero-free there, and $\Phi$ is meromorphic on a neighbourhood of $\cl\D$ with no poles on $\T$. On $\T$ the function
$$
S(\zeta):=\frac{\zeta B'(\zeta)}{B(\zeta)}=\sum_{j=1}^d\frac{1-|a_j|^2}{|\zeta-a_j|^2}\ \ge\ \sum_{j=1}^d\frac{1-|a_j|}{1+|a_j|}=:\sigma>0
$$
is real. As noted in Section~\ref{sec:WB}, $\Phi$ has no zeros in $\D$ and exactly $d$ simple poles there. Suppose that $g(\T)\cap(-\infty,-\sigma]=\emptyset$. Then $\Phi(\T)\cap(-\infty,0]=\emptyset$: if $\Phi(\zeta)=g(\zeta)+S(\zeta)\le0$ then $g(\zeta)=\Phi(\zeta)-S(\zeta)$ would be real and $\le-\sigma$. In particular $\Phi\ne0$ on $\T$, and $\Phi(\T)$ omits a ray from $0$, so $\wind_0(\Phi|_\T)=0$. But by the argument principle $\wind_0(\Phi|_\T)=0-d<0$, a contradiction. Hence some $\zeta\in\T$ has $g(\zeta)\in(-\infty,-\sigma]$, which is the first claim. For the second, Theorem~\ref{thm:upper} gives $|a_j|\ge B(0)\ge1/4$, and for $1/4\le r<1$ one has $\log\frac1r\le\frac{1-r}{r}=\frac{1+r}{r}\cdot\frac{1-r}{1+r}\le5\,\frac{1-r}{1+r}$; summing over $j$ gives $\log\frac1{B(0)}\le5\sigma\le5\nu(W)$.

(b) The first statement is Lemma~\ref{lem:wdrop} together with Remark~\ref{rem:g}. For the second, note that $\wind_0(g|_\T)=0$ because $g$ is holomorphic and zero-free on a neighbourhood of $\cl\D$. If all intersections of $g(\T)$ with $(-\infty,0)$ are transversal crossings (finitely many, by analyticity), the winding number equals the number of downward crossings minus the number of upward crossings; hence both numbers are equal, and both are $\ge1$ since the intersection is nonempty.
\end{proof}

\begin{proposition}\label{prop:increase}
Let $(W,B)$ be a regular admissible pair of degree $d\ge0$, $F=WB$, and set $u(t)=\log|F(e^{it})|$, $v(t)=\frac{d}{dt}\arg F(e^{it})$. Suppose \eqref{eq:min} holds at $n\ge1$ distinct parameters $t_1,\dots,t_n$, and let $B_r$ be as in Lemma~\ref{lem:wdrop}. Then for all $r<1$ sufficiently close to $1$, $(W,BB_r)$ is a regular admissible pair of degree $d+n$ with
$$
\frac1{B(0)B_r(0)}=\frac{r^{-n}}{B(0)}>\frac1{B(0)} .
$$
In particular, if $\Bee_d$ is attained by a regular pair satisfying this hypothesis for some $n\ge1$, then $\Bee_{d+1}>\Bee_d$.
\end{proposition}

\begin{proof}
With $h(t)=ie^{it}F'(e^{it})/F(e^{it})=u'(t)+iv(t)$, the argument principle gives $\wind_0(h)=1+N_\D(F')-N_\D(F)=1+0-(d+1)=-d$. For $r$ close to $1$ the points $re^{it_j}$ are distinct from the zeros of $B$, so $G_r:=FB_r=W(BB_r)$ has $d+n+1$ simple zeros in $\D$ and $ie^{it}G_r'(e^{it})/G_r(e^{it})=h(t)+iS_r(t)$. The proof of Lemma~\ref{lem:wdrop} (Remark~\ref{rem:F}) shows that for $r$ close to $1$, $h+iS_r$ is zero-free and $\wind_0(h+iS_r)=\wind_0(h)-n=-d-n$. Hence $G_r'\ne0$ on $\T$ and $1+N_\D(G_r')-(d+n+1)=-d-n$, i.e.\ $N_\D(G_r')=0$. Thus $(W,BB_r)$ is regular admissible of degree $d+n$, and $BB_r(0)=B(0)r^n$. Apply the construction with a single parameter $t_1$ ($n=1$) to get a regular admissible pair of degree $d+1$ with
$$
\frac{1}{B(0)B_r(0)}=r^{-1}\Bee_d>\Bee_d;
$$
hence $\Bee_{d+1}>\Bee_d$.
\end{proof}

\begin{theorem}\label{thm:attain}
For every $n\ge0$ there is an admissible pair of degree at most $n$ with $1/B(0)=\Bee_n$. For $n=1$ the extremal pair has degree exactly one.
\end{theorem}

\begin{proof}
The case $n=0$ is $(z,1)$. Fix $n\ge1$. Choose admissible pairs $(W_j,B_j)$ of degree at most $n$ with $1/B_j(0)\to\Bee_n$. Theorem~\ref{thm:upper} keeps $B_j(0)$ in $[1/4,1]$. Pass to a subsequence so that $\deg B_j$ is constant, $W_j\to W$ locally uniformly with $W\in\cS$, and so that every zero of $B_j$ converges in $\cl\D$. Let $B$ be the Blaschke product of those limit points which lie in $\D$, normalised by $B(0)>0$, or $B\equiv1$ if none remain. A normalised factor whose zero tends to $\T$ converges to $1$ uniformly on compact subsets of $\D$, and $B_j(0)\to B(0)$. Thus $W_jB_j\to WB$ locally uniformly and $(WB)'(0)=B(0)=1/\Bee_n>0$. Each $(W_jB_j)'$ is zero-free in $\D$. By Hurwitz's theorem, $(WB)'$ is either zero-free or identically zero. The second alternative contradicts the value at the origin. The limit pair is admissible of degree at most $n$ and attains $\Bee_n$.

If $n=1$ and the single zero escaped to $\T$, the limit would have $B(0)=1$. Theorem~\ref{thm:lower} produces a degree-one value strictly larger than $1$, so a maximizing sequence cannot lose its zero. The extremal pair has degree one.
\end{proof}

\begin{corollary}\label{cor:strict}
$\Bee_1<\sqrt{7/5}$, and $\Bee_d<4-\frac{529}{576}\frac{(\log2)^2}{d^2}$ for every $d\ge1$.
\end{corollary}

\begin{proof}
By Theorem~\ref{thm:attain} there is an admissible pair $(W,B)$ with $\deg B=1$ and $1/B(0)=\Bee_1$; Theorem~\ref{thm:sqrt75} gives $1/B(0)<\sqrt{7/5}$. Likewise $\Bee_d=1/B(0)$ for some admissible pair of degree $k\le d$, and Theorem~\ref{thm:upper} gives $1/B(0)<4-\frac{529}{576}\frac{(\log 2)^2}{k^2}\le4-\frac{529}{576}\frac{(\log 2)^2}{d^2}$ if $k\ge1$, while $1/B(0)=1$ if $k=0$.
\end{proof}

\begin{remark}
The above argument does \textbf{not} show that an extremal pair must be regular. Moreover for $n\ge2$ an extremal pair need not have degree exactly $n$, as a zero may escape to $\T$ in the limiting argument of the above proof.
\end{remark}

\section{Concluding remarks}\label{sec:cr}

The identities $\M=\mathcal{A}=\Bee$ place Smale's mean value conjecture in univalent
function theory. The constant $\M$ is the supremum of $1/B(0)$ over
those $W\in\cS$ for which $WB$ is locally univalent for some finite
Blaschke product $B$. However, finding the exact value of $\Bee$ is not a problem about the whole of $\cS$ because a starlike map $W$ admits no nonconstant factor $B$. Indeed, if $W$ is starlike then $\Rea g>0$ in $\D$. If $B$ is a nonconstant Blaschke product with $k\ge1$ distinct zeros, then $zB'/B\to S>0$ uniformly as $|z|\to1$, so for $r<1$ close to $1$ we have $\Rea\Phi=\Rea g+\Rea(zB'/B)>0$ on $|z|=r$ and hence $\wind_0(\Phi|_{r\T})=0$. On the other hand, the argument principle gives $\wind_0(\Phi|_{r\T})=0-k<0$ if $\Phi$ were zero-free. For $W$ univalent on a neighbourhood of $\cl\D$, Theorem~\ref{thm:char} says much more: $W$ carries a nonconstant $B$ only if $g(\T)$ meets $(-\infty,0)$, and it does carry one if $g(\T)$ crosses $(-\infty,0)$ transversally. Therefore, it is not clear if $\Bee =4$. On the other hand, for the restricted
constants $\Bee_d$, Theorem~\ref{thm:nonempty} shows that every degree
occurs, and Theorem~\ref{thm:attain} shows that each supremum $\Bee_d$
is attained. Moreover, Theorem \ref{thm:upper} and Corollary~\ref{cor:strict} give $\Bee_d < 4-O(1/d^2)$ and we also have
$$
1<\Bee_1\le\Bee_2\le\cdots\le\Bee=\M\le4,
$$
with the degree-one window
$$
\sqrt{\frac{8081}{6656}}
\le\Bee_1
< \sqrt{\frac{7}{5}}.
$$
Whether $\Bee_d$ is strictly increasing, whether $\Bee_1=\Bee$, and
whether $\M=4$, are open. The value $4$, if it is the supremum, can be
approached only with $\deg B_n\to\infty$. The lobe pairs of
Theorem~\ref{thm:nonempty} prove existence in every degree but give
$1/B(0)$ only slightly larger than $1$, so they do not separate
$\Bee_n$ from $\Bee_1$. Earlier degree-one pairs, built from a truncated
odd Koebe function and from a bent radial slit, gave only
$\sqrt{70837/67744}=1.02257\ldots$ and $\sqrt{61/58}=1.02553\ldots$.
Both are superseded by the quartic polynomial construction in section \ref{sec:quartic} and are omitted. For the same reason, we do not include the estimate $\cA>1.00114$. It would be interesting to see if one can get a larger lower bound of $\M$ through $\cA$.

The constant $\Bee$ records a supremum of $1/B(0)$ over pairs for which
$(WB)'$ does not vanish on $\D$. If one instead retains those pairs for
which $(WB)'$ does vanish somewhere in $\cl\D$, the natural quantity is
the largest ratio of $B$ at a critical point of the product,
$$
\Delta(W,B)
:=\max_{\substack{u\in\cl\D\\ (WB)'(u)=0}}
\frac{|B(u)|}{|B(0)|}.
$$
The maximum runs over pairs for which $W$ and $B$ extend holomorphically through $\cl\D$, so that it is taken on a compact set. Holomorphic extension through the closed disk is not the regularity condition in Definition~\ref{def:admissible}: that definition also requires $(WB)'\ne0$ on $\cl\D$, which is exactly the condition excluded here. The dual constant
$$
\Bee^\vee
:=\inf\bigl\{\Delta(W,B):(WB)'\text{ has a zero in }\cl\D\bigr\}
$$
asks how small that largest ratio can be made; write $\Bee^\vee_k$ for
the same infimum restricted to pairs with $\deg B\le k$. One motivation
is the dual mean value constants of
\cite{DubininSugawa,NgZhang,JN26}:
$$
D(P):=\max_{P'(c)=0}\Bigl|\frac{P(c)}{c}\Bigr|,
\qquad
\cD_d:=\inf_{P\in\cP_d}D(P),\qquad \cD:=\inf_{P\in\cP}D(P).
$$
The comparison $\frac{(d-1/2)^{1/d}}{d^2}\le\cD_d\le1/d$
(see \cite{JN26}) gives $\cD=0$. The upper bound is elementary: the
polynomial $P_d(z)=\bigl((1+z)^d-1\bigr)/d$ lies in $\cP_d$ and has
$D(P_d)=1/d$. In \cite{JNDual}, we show that $\Bee^{\vee}_1=\sqrt{5}-2$ and $\Bee^{\vee}_d=0$ for $d\ge 2$ and hence $\Bee^\vee=\cD=0$.

\appendix
\section{The Schur--Cohn polynomials}
The polynomials below are the primitive integer forms of the Schur data for the kernel of $H$, scaled by $65536$.
Each has positive constant term.
Sturm's theorem finds no root of $D_1$, $D_2$, $D_3$ or $L_1$ in $(-1,1)$.
\begin{align*}
D_1(c)
&=114165512497
-22817013760\,c^2
+91268055040\,c^4
-91268055040\,c^6\\
&=114165512497-22817013760\,c^2(1-2c^2)^2.
\end{align*}
One has $D_1(0)=114165512497$ and $D_1(\pm1)=91348498737$.
\begin{align*}
D_2(c)
&=9063713835607204322145
+40354798711500779814912\,c^2\\
&\quad
-137255167184223106760704\,c^4
+134505322446801191567360\,c^6\\
&\quad
-36580972445627842560000\,c^8
-16659715741568938803200\,c^{10}\\
&\quad
+8329857870784469401600\,c^{12}.
\end{align*}
One has $D_2(0)=9063713835607204322145$ and
$$
D_2(\pm1)=1757837493273756982113.
$$

The polynomial $D_3$ is even of degree $24$.
Its coefficients of $c^{2k}$, for $k=0,\ldots,12$, are
\begin{align*}
A_0&=82150908493777459655963563840981568937401025,\\
A_1&=-781613022560571415504824817236223570316293120,\\
A_2&=3359768867334789754216043264686014645974073344,\\
A_3&=-8762583793491993307205788346457380702461624320,\\
A_4&=15902598393884913304241897372883275148664242176,\\
A_5&=-22036434677925167143977401224218413247624118272,\\
A_6&=24750570505604861462796289357056323903473844224,\\
A_7&=-22753585497242017840835179977281280929525923840,\\
A_8&=16465293459460144484786767659174979609684869120,\\
A_9&=-8793161256645340595516850795954044879010529280,\\
A_{10}&=3198169252593861067905186904351732202156851200,\\
A_{11}&=-700543296885997489235622760559064884707328000,\\
A_{12}&=69386532147469974137999293426046662082560000.
\end{align*}
Thus $D_3(c)=\sum_{k=0}^{12}A_k c^{2k}$.
The values $D_3(0)=A_0$ and
$$
D_3(\pm1)=16374768729715464479493712945527328024257
$$
are positive, and the Sturm count on $(-1,1)$ is zero.

The squared modulus of the leading coefficient of the first transform is a positive rational multiple of
\begin{align*}
L_1(c)
&=60578161744620649
-695261090860348232\,c^2
+2420274533960558224\,c^4\\
&\quad
-2433923033732382720\,c^6
+748836658505318400\,c^8.
\end{align*}
One has $L_1(0)=60578161744620649$ and $L_1(\pm1)=100505229617766321$, both positive, and $L_1$ has no root in $(-1,1)$. 

The squared modulus of the next leading coefficient is a positive rational multiple of $c^2 M(c)$, where
\begin{align*}
M(c)
&=413083172164177336605103112162967027712000\,c^{20}\\
&\quad-3447316139158502872422850731861552712908800\,c^{18}\\
&\quad+11965662987622824100548522121995938923806720\,c^{16}\\
&\quad-21382225448365818984329612061175363749806080\,c^{14}\\
&\quad+17732888791847426272417022333948048609116160\,c^{12}\\
&\quad+2138052913931431925926901100320534009217024\,c^{10}\\
&\quad-17715917688319063451495510502968611613704192\,c^{8}\\
&\quad+12821419940103179304908612472951512414289920\,c^{6}\\
&\quad+120126962767902603097970472150086924107776\,c^{4}\\
&\quad-4120451371017080615122578837835904128010240\,c^{2}\\
&\quad+1477677458388396694474675376462469384142525
\end{align*}
and
\begin{align*}
M(0)
&=1477677458388396694474675376462469384142525,\\
M(\pm1)
&=3001579964872314608254856150125087962813
\end{align*}
are positive.
The polynomial $c^2 M(c)$ has exactly one distinct root in $(-1,1)$, namely $c=0$, so $M$ has none.


\begin{thebibliography}{99}

\bibitem{Lean}
T.~Adamczewski (repository),
\emph{Lean 4 disproof of Smale's mean value conjecture ($K=1$), found by
GPT-6 Astra},
\url{https://github.com/tadamcz/mean-value-problem}, 2026.

\bibitem{BMN}
A.~F. Beardon, D.~Minda, T.~W. Ng,
Smale's mean value conjecture and the hyperbolic metric,
\emph{Math.\ Ann.} \textbf{322} (2002), 623--632.

\bibitem{cheung2011}
P.~Cheung, ``Smale's inequalities for polynomials and mean value conjecture,'' M.Phil. thesis, The University of Hong Kong, 2011. \url{https://hub.hku.hk/handle/10722/144814}.


\bibitem{choi}
Y.~Choi, ``Residual Julia sets of Newton's maps and Smale's problems on the efficiency of Newton's method,'' M.Phil. Thesis, The University of Hong Kong, 2006. (\url{https://hub.hku.hk/handle/10722/51861})

\bibitem{Crane1}
E.~Crane, Extremal polynomials in Smale’s mean value conjecture, \emph{Computational Methods and Function Theory} \textbf{6} (1) (2006): 145--163.

\bibitem{Crane}
E.~Crane,
A bound for Smale's mean value conjecture for complex polynomials,
\emph{Bull.\ London Math.\ Soc.} \textbf{39} (2007), 781--791.

\bibitem{Dieudonne}
J.~Dieudonn\'e,
Recherches sur quelques probl\`emes relatifs aux polyn\^omes et aux fonctions born\'ees d'une variable complexe,
\emph{Ann.\ Sci.\ \'Ecole Norm.\ Sup.} (3) \textbf{48} (1931), 247--358.

\bibitem{Dubinin2006}
V.~N. Dubinin,
Inequalities for critical values of polynomials,
\emph{Sb.\ Math.} \textbf{197} (2006), no.~8, 1167--1176;
translated from \emph{Mat.\ Sb.} \textbf{197} (2006), no.~8, 63--72.

\bibitem{Duren}
P.~L. Duren, \emph{Univalent Functions},
Grundlehren der math.\ Wissenschaften 259, Springer, New York, 1983.

\bibitem{DubininSugawa}
V.~N. Dubinin, T.~Sugawa,
Dual mean value problem for complex polynomials,
\emph{Proc.\ Japan Acad.\ Ser.\ A Math.\ Sci.} \textbf{85} (2009), 135--138.

\bibitem{FLW}
S. Fan, L. Liao, \& Y. Wang, . A dynamical approach to Smale's mean value conjecture. \emph{Journal of Complexity}, {\bf 98} (2027), 102094.


\bibitem{Grunsky1932}
H.~Grunsky,
Neue Absch\"atzungen zur konformen Abbildung ein- und mehrfach
zusammenh\"angender Bereiche,
\emph{Schriften Math.\ Sem.\ Inst.\ Angew.\ Math.\ Univ.\ Berlin}
\textbf{1} (1932), 95--140.

\bibitem{Hilbert}
D.~Hilbert,
\"Uber die Entwickelung einer beliebigen analytischen Function einer
Variabeln in eine unendliche nach ganzen rationalen Functionen
fortschreitende Reihe,
\emph{G\"ottinger Nachrichten} (1897), 63--70.

\bibitem{JN26}
A.~Jatar, T.~W. Ng,
Smale's mean value conjecture and its dual conjecture for complex
polynomials, arXiv:2608.27047 (2026).
\bibitem{JNDual}
A.~Jatar, T.~W. Ng,
Dual mean value constant and Blaschke-critical univalent functions.

\bibitem{Loewner1923}
K.~L\"owner,
Untersuchungen \"uber schlichte konforme Abbildungen des Einheitskreises. I,
\emph{Math.\ Ann.} \textbf{89} (1923), 103--121.

\bibitem{MaMinda}
W.~Ma, D.~Minda,
Two-point distortion theorems for bounded univalent functions,
\emph{Ann.\ Acad.\ Sci.\ Fenn.\ Math.} \textbf{22} (1997), 425--444.


\bibitem{NgZhang}
T.~W. Ng, Y.~Zhang,
Smale's mean value conjecture for finite Blaschke products,
\emph{J.\ Anal.} \textbf{24} (2016), 331--345.

\bibitem{NT}
B. Nagy, V. Totik, Sharpening of Hilbert’s lemniscate theorem. \emph{Journal d'Analyse Mathématique}, \textbf{96} (1)(2005), 191--223.


\bibitem{Pommerenke}
Ch.~Pommerenke, \emph{Boundary Behaviour of Conformal Maps},
Grundlehren der math.\ Wissenschaften 299, Springer, Berlin, 1992.

\bibitem{RS}
Q.~I. Rahman, G.~Schmeisser, \emph{Analytic Theory of Polynomials},
London Math.\ Soc.\ Monographs 26, Oxford Univ.\ Press, 2002.

\bibitem{Re} R. Remmert, {\it Classical Topics in Complex Function Theory}.
Springer-Verlag, New York, 1998.

\bibitem{Rudin}
W.~Rudin, \emph{Real and Complex Analysis}, 3rd ed.,
McGraw-Hill, New York, 1987.

\bibitem{Sheil-Small}
T. Sheil-Small, \emph{Complex polynomials}. Cambridge University Press, 2002.

\bibitem{Smale}
S.~Smale,
The fundamental theorem of algebra and complexity theory,
\emph{Bull.\ Amer.\ Math.\ Soc.} \textbf{4} (1981), 1--36.

\bibitem{Smale2000}
S.~Smale, Mathematical problems for the next century, \emph{Mathematics: Frontiers and Perspectives: Frontiers and Perspectives}. page 271--294, American Mathematical Soc., 2000.








\end{thebibliography}
\end{document}